\documentclass[11pt]{article}
\usepackage[utf8]{inputenc}
\usepackage[T1]{fontenc}
\usepackage[english]{babel}
\usepackage{csquotes}

\usepackage{amsmath}
\usepackage{amsfonts}
\usepackage{amssymb}

\usepackage{quiver}

\usepackage[bbgreekl]{mathbbol}
\DeclareMathSymbol\bbDelta  \mathord{bbold}{"01}

\usepackage{extarrows}

\usepackage{titlesec}
\usepackage{verbatim}
\usepackage{enumitem}
\usepackage{geometry}
\usepackage{amsthm}
\usepackage{fancyvrb}
\usepackage{todonotes}
\usepackage{mathtools}
\tikzcdset{scale cd/.style={every label/.append style={scale=#1},
    cells={nodes={scale=#1}}}}
\newtheoremstyle{theoremdd}
{\topsep}
{\topsep}
{\itshape}
{0pt}
{\bfseries}
{.}
{ }
{\thmname{#1}\thmnumber{ #2}\textnormal{\thmnote{ (#3)}}}
\theoremstyle{theoremdd}
\newtheorem{theorem}{Theorem}[section]
\newtheorem{definition}[theorem]{Definition}

\newtheorem{propdef}[theorem]{Proposition-definition}
\newtheorem{thdef}[theorem]{Theorem-definition}
\newtheorem{corollary}[theorem]{Corollary}
\newtheorem{lemma}[theorem]{Lemma}

\newtheoremstyle{theoremddnoproof}
{\topsep}
{\topsep}
{\itshape}
{0pt}
{\bfseries}
{.}
{ }
{\thmname{#1}\thmnumber{ #2}\textnormal{\thmnote{ (#3)}}}
\theoremstyle{theoremddnoproof}

\theoremstyle{definition}
\newtheorem{remark}[theorem]{Remark}

\usepackage[colorlinks=true, citecolor=blue, linkcolor=blue, urlcolor=blue]{hyperref}
\usepackage{cleveref}
\usepackage[backend=biber, maxbibnames=99]{biblatex}
\usepackage{colonequals}

\usepackage{xcolor}
\usepackage{xspace}

\usepackage[frozencache=true, cachedir=minted-cache]{minted}

\usepackage{fontspec}

\newcommand{\iso}{\stackrel{\sim}{\smash{\rightarrow}\rule{0pt}{0.4ex}}}

\newcommand{\mathlib}{\texttt{Mathlib}\xspace}

\definecolor{codegray}{gray}{0.9}

\newmintinline[lean]{lean4}{bgcolor=codegray}
\newminted[leancode]{lean4}{bgcolor=codegray, fontsize=\footnotesize}
\title{Unbiasing symmetric monoidal categories in Lean}

\newcommand{\mathliburl}{https://github.com/leanprover-community/mathlib4}
\newcommand{\repourl}{https://github.com/robin-carlier/SymmMonCoherence}

\newcommand{\refCommitMathlib}{v4.28.0} 
\newcommand{\refCommitRepo}{2070e1d536854a49d4c4f98cd73857ba40c4cb92} 
\newcommand{\mathlibLinkRange}[4]{%
    \href{\mathliburl/blob/\refCommitMathlib/#1\#L#2-L#3}{#4}%
}
\newcommand{\smcLinkRange}[4]{%
    \href{\repourl/blob/\refCommitRepo/#1\#L#2-L#3}{#4}%
}
\newcommand{\mathlibLink}[3]{%
    \href{\mathliburl/blob/\refCommitMathlib/#1\#L#2}{#3}%
}
\newcommand{\smcLink}[3]{%
    \href{\repourl/blob/\refCommitRepo/#1\#L#2}{#3}%
}
\newcommand{\mathlibPR}[1]{%
    \href{\mathliburl/pull/#1}{\##1}%
}
\author{Robin Carlier\footnote{The author acknowledges support
		from the ANR project ``HQDIAG'', ANR-21-CE40-0015.}}

\date{February 28, 2026}
\begin{document}
\maketitle
\begin{abstract}
We present a formalization in Lean~4, within the framework of the mathematical library \mathlib, of the unbiasing process
for symmetric monoidal categories. This is realized by extending the data of a symmetric monoidal category to a
\(\mathrm{Cat}\)-valued pseudofunctor from the (2,1)-category of spans of finite sets,
encoding tensor products of higher arities and their coherences.
The construction relies on a formalization of Mac Lane's coherence theorem using
Piceghello's presentation of free symmetric monoidal categories as symmetric lists,
and uses an encoding of universal formulas via an appropriate Kleisli bicategory.
\end{abstract}
\section{Introduction}

When setting up the basic algebraic hierarchy for mathematics, one defines a commutative monoid as a set \(M\)
equipped with the data of a unit element \(0_M\) and a commutative and associative composition law
\(\cdot +_M \cdot : M \times M \rightarrow M \). Given elements \(a, b, c, d\) of such a commutative monoid \(M\),
the data allows us to make formal sense of expressions such as \(\left(a +_M b\right) +_M \left(c +_M d\right)\),
\(\left(a +_M \left(c +_M b\right)\right) +_M d\) or \(b +_M \left(a +_M \left(d +_M c\right)\right)\),
and the properties satisfied by the function \( +_M \) let us prove that all of these expressions are equal.
For mathematical applications, it very quickly becomes important to extend the binary operation to operators with
higher arities. For instance, a formula like
\[\sum\limits_{\rho \in \mathrm{Irr}(G)} (\mathrm{dim}(\rho))^2 = \mathrm{Card}(G)\]
for a finite group \(G\)
involves a summation operator that is defined over an arbitrary finite set of elements of the underlying monoid
without a canonical order.
The mere existence of such an operator requires a proof that it is well-defined, i.e., that the resulting value
cannot depend on how the underlying list of summands is enumerated or associated.

In the context of software proof assistants, libraries of formalized mathematics, such as \mathlib~\cite{mathlib2020}\cite{githubmathlib4}, have to go through
that process in order to give their users a framework for working with such finite sums.
For instance, \mathlib has a \mathlibLinkRange{Mathlib/Algebra/BigOperators/Group/Multiset/Defs.lean}{38}{44}{declaration}
\begin{leancode}
def Multiset.sum {M : Type u} [AddCommMonoid M] : Multiset M → M := …
\end{leancode}
that realizes the sum over a multiset of elements of \lean{M} by using the definition of
\lean{Multiset M} as a quotient of the type \lean{List M} modulo permutations. The declaration
\lean{Finset.sum} then builds on that definition to provide unbiased sums to users.\footnote{
These sums still bundle a specific choice regarding associativity, coming from the fact that
\lean{List.sum} is defined as folding the sum operator on the left. The irrelevance of that choice is then proven
in e.g., \lean{Finset.sum_disjUnion}. Following the thread of proofs, it reduces to
\lean{Multiset.fold_add}, which is where associativity ends up being used.}

\subsection*{From commutative monoids to symmetric monoidal categories}

The natural categorification of commutative monoids is symmetric monoidal categories.
Symmetric monoidal categories were independently introduced by Saunders Mac Lane \cite{Maclane1963NaturalAA} as an
abstraction of the structure shared by the categorical products in categories with products,
the tensor product on modules and the tensor product of chain complexes.
This well-studied concept is now considered a classical example of one of the many
\emph{pseudoalgebraic} structures one can put on a category.

The classical definition of a symmetric monoidal category (\cite{Maclane1963NaturalAA}, \cite[Def. 6.1.2]{Borceux_1994}) consists of
a category \( \mathsf{C} \) equipped with a tensor product bifunctor
\( \otimes_\mathsf{C} : \mathsf{C} \times \mathsf{C} \to \mathsf{C} \), a unit object \( \mathbb{1}_{\mathsf{C}} \) and
  for all objects \(X, Y, Z\) of \(\mathsf{C}\), natural isomorphisms\begin{align*}
    \alpha_{X,Y,Z}^{\mathsf{C}} &: (X \otimes_{\mathsf{C}} Y) \otimes_\mathsf{C} Z \iso X \otimes_{\mathsf{C}} (Y \otimes_{\mathsf{C}} Z) \\
  \lambda_{X}^{\mathsf{C}} &: \mathbb{1}_{\mathsf{C}} \otimes_{\mathsf{C}} X \iso X \\
  \rho_{X}^{\mathsf{C}} &: X \otimes_{\mathsf{C}} \mathbb{1}_{\mathsf{C}} \iso X \\
  \beta_{X, Y}^{\mathsf{C}} &: X \otimes_{\mathsf{C}} Y \iso Y \otimes_{\mathsf{C}} X \\
\end{align*}
respectively called associators, left unitors, right unitors, and braidings, such that the
following diagrams commute:

\[\begin{tikzcd}[cramped]
	& {((X \otimes_{\mathsf{C}} Y)\otimes_{\mathsf{C}} Z) \otimes_{\mathsf{C}} T} \\
	{(X \otimes_{\mathsf{C}} Y)\otimes_{\mathsf{C}} (Z \otimes_{\mathsf{C}} T)} && {(X \otimes_{\mathsf{C}} (Y \otimes_{\mathsf{C}} Z)) \otimes_{\mathsf{C}} T} \\
	\\
	{X \otimes_{\mathsf{C}} (Y \otimes_{\mathsf{C}} (Z \otimes_{\mathsf{C}} T))} && {X \otimes_{\mathsf{C}} ((Y \otimes_{\mathsf{C}} Z) \otimes_{\mathsf{C}} T)}
	\arrow["{\alpha^{\mathsf{C}}_{X \otimes_{\mathsf{C}} Y,Z,T}}"', from=1-2, to=2-1]
	\arrow["{\left(\alpha^{\mathsf{C}}_{X,Y,Z}\right) \otimes_{\mathsf{C}} T}", from=1-2, to=2-3]
	\arrow["{\alpha^{\mathsf{C}}_{X,Y,Z \otimes_{\mathsf{C}} T}}"', from=2-1, to=4-1]
	\arrow["{\alpha^{\mathsf{C}}_{X,Y \otimes_{\mathsf{C}} Z, T}}", from=2-3, to=4-3]
	\arrow["{X \otimes_{\mathsf{C}} \left(\alpha^{\mathsf{C}}_{Y,Z,T}\right)}", from=4-3, to=4-1]
\end{tikzcd}\]

\[\begin{tikzcd}[cramped]
	{(X \otimes_{\mathsf{C}} \mathbb{1}_{\mathsf{C}}) \otimes_{\mathsf{C}} Y} && {X \otimes_{\mathsf{C}} (\mathbb{1}_{\mathsf{C}} \otimes_{\mathsf{C}} Y)} \\
	& {X \otimes_{\mathsf{C}} Y}
	\arrow["{\alpha_{X, \mathbb{1}_{\mathsf{C}}, Y}^{\mathsf{C}}}", from=1-1, to=1-3]
	\arrow["{\rho^{\mathsf{C}}_X \otimes_{\mathsf{C}} Y}"'{pos=0.2}, from=1-1, to=2-2]
	\arrow["{X \otimes_{\mathsf{C}} \lambda^{\mathsf{C}}_Y}"{pos=0.1}, from=1-3, to=2-2]
\end{tikzcd}\]

\[\begin{tikzcd}[cramped]
	&&& {(X \otimes_{\mathsf{C}} Y) \otimes_{\mathsf{C}} Z} \\
	{X \otimes_{\mathsf{C}} Y} & {Y \otimes_{\mathsf{C}} X} & {X \otimes_{\mathsf{C}} (Y \otimes_{\mathsf{C}} Z)} && {(Y \otimes_{\mathsf{C}} X) \otimes_{\mathsf{C}} Z} \\
	& {X \otimes_{\mathsf{C}} Y} & {(Y \otimes_{\mathsf{C}} Z) \otimes_{\mathsf{C}} X} && {Y \otimes_{\mathsf{C}} (X \otimes_{\mathsf{C}} Z)} \\
	&&& {Y \otimes_{\mathsf{C}} (Z \otimes_{\mathsf{C}} X)}
	\arrow["{\alpha^{\mathsf{C}}_{X, Y, Z}}"', from=1-4, to=2-3]
	\arrow["{\beta_{X, Y}^{\mathsf{C}} \otimes_{\mathsf{C}} Z}", from=1-4, to=2-5]
	\arrow["{ \beta^{\mathsf{C}}_{X, Y}}", from=2-1, to=2-2]
	\arrow["{\mathrm{Id}_{X \otimes_{\mathsf{C}} Y}}"', from=2-1, to=3-2]
	\arrow["{ \beta^{\mathsf{C}}_{Y, X}}", from=2-2, to=3-2]
	\arrow["{\beta^{\mathsf{C}}_{X, Y \otimes_{\mathsf{C}} Z}}"', from=2-3, to=3-3]
	\arrow["{\alpha_{Y, X, Z}^{\mathsf{C}}}", from=2-5, to=3-5]
	\arrow["{\alpha^{\mathsf{C}}_{Y, Z, X}}"', from=3-3, to=4-4]
	\arrow["{Y \otimes_{\mathsf{C}} \beta_{X, Z}^{\mathsf{C}}}", from=3-5, to=4-4]
\end{tikzcd}.\]

Unlike the case of commutative monoids, associativity and commutativity in symmetric monoidal categories are
expressed as a \emph{structure} instead of a \emph{property}.
This is part of the standard philosophy of category theory: equality of objects in a category is a
notion that is not invariant under equivalences of categories and one should
instead replace the property of being equal with the structure of a chosen isomorphism. The properties appear
one categorical level higher as commutativity of diagrams involving the isomorphisms.

The presentation above gives a rather easy and straightforward way to define symmetric
monoidal categories: only low-arity data and coherences need to be supplied. On the other hand, similarly to the case
of commutative monoids, some applications
require an unbiased point of view on symmetric monoidal categories and require availability of tensor products and
coherences in every possible arity.
For instance, given a symmetric monoidal category \( \mathsf{C} \) with countable coproducts and such that the tensor
product of \( \mathsf{C} \) preserves suitable colimits in each variable,
the free commutative monoid on an object \( c\in \mathsf{C} \) can be identified with an \(\mathbb{N}\)-indexed
coproduct of objects \( \mathrm{Sym}_n^{\mathsf{C}}(c) \) where \( \mathrm{Sym}_n^{\mathsf{C}} \) sends an object
\( a \) to the colimit of the \( \mathrm{B}\Sigma_n \)-indexed diagram in \( \mathsf{C} \) that corresponds to the object
\( a^{\otimes n} \) with its natural \( \Sigma_n \) action.
This fact cannot even be stated unless the binary
tensor product of \( \mathsf{C} \) and its symmetries have been extended to higher arities.

This formality is usually resolved by invoking Mac Lane’s celebrated coherence theorem \cite[Thm. 4.2]{Maclane1963NaturalAA}
which loosely states
that there is only one ``canonical'' isomorphism made out of associators and unitors between two iterated tensor
products that only differ by bracketing, and only one ``canonical'' morphism attached to a given permutation of the variables
of an iterated tensor product. We are being intentionally vague with the statement here as we will be giving a more precise formulation
in the course of this paper (see Theorem~\ref{coherence_theorem}).

\subsubsection*{Symmetric monoidal higher categories}
Another very important application of unbiased symmetric monoidal categories comes in the form of the higher-categorical
point of view on symmetric monoidal categories. The definition of symmetric monoidal categories we gave in the previous
paragraph does not scale very well as we increase the categorical level: the triangle, pentagon, hexagon and symmetry
identities are about equalities of (1-)morphisms in \( \mathsf{C} \); if \( \mathsf{C} \) is replaced by a bicategory,
then all these equalities need to be replaced by the data of suitable 2-isomorphisms in \( \mathsf{C} \)
and new coherence identities need to be satisfied by these 2-isomorphisms (see e.g.,
\cite[Appendix C]{schommerpries2014} for an algebraic description of symmetric monoidal bicategories).
To further convince oneself that this algebraic approach can hardly scale further in practice, one can look
at the definition of a (non-symmetric) monoidal tricategory as a one-object tetracategory in
\cite{hoffnung2013spans2categoriesmonoidaltricategory}.

In the context of the theory of \((\infty,1)\)-categories, several equivalent notions of symmetric monoidal
\((\infty,1)\)-categories have been studied:
Lurie’s \( \infty \)-operads \cite{HA},
Moerdijk and Weiss's dendroidal sets \cite{Moerdijk2007},
Cranch's Lawvere symmetric monoidal \((\infty,1)\)-categories \cite{cranch2010}, to cite only a few.

All of the models mentioned above are similar in that the data of a symmetric monoidal
\((\infty,1)\)-category is interpreted (either directly, or via suitable comparison theorems) as a pseudofunctor to the
\((\infty, 1)\)-category of \((\infty, 1)\)-categories
(or, equivalently, fibrations), with source a diagram category that contains already higher arity data and symmetries:
the category of pointed finite sets (a.k.a. Segal’s category \(\Gamma\)) for \(\infty\)-operads,
the category of rooted trees \( \Omega \) in the dendroidal sets setting and the (2,1)-category of spans of finite sets for
Cranch’s Lawvere symmetric monoidal \((\infty,1)\)-categories. In all these approaches, these source categories encode
all the diagrams that are expected to ``commute'' in the resulting symmetric monoidal \((\infty,1)\)-category
and packaging the data as an \(\infty\)-functor to \(\mathrm{Cat}_\infty\) realizes the diagrams.

One of the important aspects of \((\infty,1)\)-category theory is that, when a notion is applied to objects that come
from ordinary categories, it reduces in almost all cases to the corresponding notion in ordinary category theory:
for instance, when viewing an ordinary category $\mathsf{C}$ as an \((\infty,1)\)-category,
the \(\infty\)-groupoid of morphisms from \( x \) to \(y\) is isomorphic to the set of morphisms from
\( x \) to \( y \) in \( \mathsf{C} \) when viewing sets as discrete \(\infty\)-groupoids.
Similarly, \((\infty,1)\)-categorical notions of limits and colimits reduce to ordinary limits and colimits when the
source and targets of the diagrams are ordinary categories.
Symmetric (and non-symmetric) monoidal categories are no exception to this principle:
when viewed as an \((\infty,1)\)-category, the data of a symmetric monoidal ordinary category is equivalent to
the data of a symmetric monoidal \((\infty,1)\)-category.
Such a statement cannot avoid the ``unbiasing'' step that extends the binary tensor product to any finite family of
objects with suitable inner symmetries. Mac Lane’s coherence theorem is usually mentioned when this step occurs, although its invocation may not always be very explicit (\cite[Construction 2.0.0.1, (ii) and (iii)]{HA}).
We refer the reader to \cite{JARDINE1991103} for one of the few places in the literature where the full process of
unbiasing symmetric monoidal categories is actually spelled out, though not in terms of higher categories.

\subsubsection*{State of current formalizations in Lean}

In the Lean~4 proof assistant, and more specifically in the mathematics library \mathlib, symmetric monoidal categories
have been \mathlibLinkRange{Mathlib/CategoryTheory/Monoidal/Braided/Basic.lean}{361}{367}{formalized} using the classical biased approach. This is very reasonable,
as this is the most elementary approach (\mathlib's definition could be performed in a file
that imports nothing but the definition of categories and the definition of isomorphisms)
as well as the one that leads to constructors with the least data and proof obligations.
This ease of definition comes at the price that, while elementary manipulation can be performed in a nice way,
some more complex constructions cannot easily be performed, for instance, the description of free commutative monoid
objects we mentioned earlier.
Another reason to seek an unbiased definition is to provide the necessary ``glue'' for the higher-categorical
perspectives mentioned previously. It seems likely that \mathlib will eventually include foundational results for
higher categories, making this connection essential.
This feeling is based on the observation that \mathlib already
has definitions relevant to quasi-categorical foundations of the theory of higher categories such as
\mathlibLink{Mathlib/AlgebraicTopology/Quasicategory/Basic.lean}{39}{quasi-categories},
\mathlibLink{Mathlib/AlgebraicTopology/SimplicialNerve.lean}{215}{simplicial nerves of simplicial categories},
a growing amount of simplicial homotopy theory,
as well as the existence of an ongoing project to formalize Riehl-Verity’s theory of \(\infty\)-cosmoi and
formal category theory inside an \(\infty\)-cosmos.
As this theory develops,
it will become important to be able to reinterpret objects coming from the already existing
formalization of ordinary category theory that is available in \mathlib as the corresponding higher-categorical objects.

The coherence theorem for (non-symmetric) monoidal categories is already formalized in \mathlib, but no formal unbiasing
of even non-symmetric monoidal categories is available.

\subsubsection*{Results}

Borrowing terminology from \cite[Def. 00AL]{kerodon}, we will call the \emph{pith} of a bicategory \( \mathsf{B} \) the bicategory
that is obtained by discarding non-invertible 2-morphisms in \(\mathsf{B}\); we will denote it by \(\mathrm{Pith}(\mathsf{B})\).
Given a category \( \mathsf{C} \) with pullbacks, we will denote by \(\mathrm{Span}(\mathsf{C})\) the bicategory of spans in \(\mathsf{C}\).

The main result that our project formalizes is the following:
\begin{theorem}\label{main_th}
  Let \(\left( \mathsf{C}, \otimes_{\mathsf{C}}, \mathbb{1}_{\mathsf{C}}, \alpha^{\mathsf{C}}, \lambda^{\mathsf{C}},
  \rho^{\mathsf{C}}, \beta^{\mathsf{C}}\right) \) be a symmetric monoidal category.
  There exists a pseudofunctor
  \[\mathsf{C}^{\otimes} : \mathrm{Pith}\left(\mathrm{Span}\left({\mathrm{Fin}}\right)\right) \to \mathrm{Cat}\]
  that sends a finite set \( J \) to the product category \( \mathsf{C}^{J} \), and sends a span
  \[\begin{tikzcd}
    & {A} \\
    {J} && {K}
    \arrow["{f}"', from=1-2, to=2-1]
    \arrow["{g}", from=1-2, to=2-3]
  \end{tikzcd}\]
  to the functor \(g_!f^*\). These functors are defined as follows:
  \begin{itemize}
    \item The functor \(f^* : \mathsf{C}^{J} \to \mathsf{C}^{A} \) is precomposition along \( f \).
    \item The functor \(g_! : \mathsf{C}^{A} \to \mathsf{C}^{K} \) sends an
      \( A \)-indexed family \( \left(x_a\right)_{a \in A} \) to the \(K\)-indexed family
      \[
        \left(\bigotimes\limits_{a \in g^{-1}\left(\{k\}\right)} x_a \right)_{k \in K}.
      \]
  \end{itemize}
\end{theorem}
This theorem expresses that a symmetric monoidal category as defined in the usual biased version defines a
Lawvere symmetric monoidal \((\infty,1)\)-category (or rather, an appropriate ordinary-categorical version of this)
in the terminology of \cite{cranch2010}.
\begin{remark}
Some readers might be surprised by our choice of stating this result in terms of Cranch's Lawvere symmetric
monoidal categories instead of Lurie's \(\infty\)-operads, the latter being the most used model in the
\((\infty,1)\)-categorical literature.
  The full comparison theorem stating that the \((\infty,1)\)-categories of Lawvere symmetric monoidal
    \((\infty,1)\)-categories and Lurie-style symmetric monoidal \((\infty,1)\)-categories are equivalent
    requires some work (see \cite[Prop. C.1]{BACHMANN2021} for a rather concise proof).
    Fortunately, the easiest direction of the comparison result is the one that turns a Lawvere symmetric
    monoidal \((\infty,1)\)-category into a Lurie-style symmetric monoidal \((\infty,1)\)-category: this is achieved by
    precomposition along an explicitly defined pseudofunctor from \(\mathrm{Fin}_*\) (seen as a locally discrete
    bicategory) to \(\mathrm{Pith}\left(\mathrm{Span}\left({\mathrm{Fin}}\right)\right)\), followed by taking the
    Grothendieck construction (which is the easy direction of the straightening/unstraightening equivalence).
    Hence, our result would easily translate into the construction of a pseudofunctor out of the category of pointed
    finite sets.
    On the other hand, constructing something out of the pith of the bicategory of spans of finite sets
    from the data of a suitable pseudofunctor out of pointed finite sets
    would be more challenging, as this would involve performing a right Kan extension of a pseudofunctor
    along another pseudofunctor rather than restricting it.
\end{remark}
We believe that this work constitutes the first mechanized proof of the statement in Theorem~\ref{main_th}.

As announced earlier in this introduction, a primary ingredient in proving the above theorem is the coherence theorem for
symmetric monoidal categories. We implement the following form of this theorem:

\begin{theorem}\label{coherence_theorem}
  Let \(J\) be a set. The free symmetric monoidal category on \(J\) is equivalent (as a symmetric monoidal category)
  to the category \(\mathrm{Core}\left(\mathrm{Fin}_{/J}\right)\) of finite sets and bijections over
  \(J\) endowed with the symmetric monoidal structure induced by the cocartesian symmetric monoidal structure
  on sets.
\end{theorem}

S. Piceghello previously formalized part of the coherence theorem in the setting of Homotopy Type Theory in \cite{Piceghello} by identifying the free symmetric monoidal category on a type with the groupoid of symmetric lists,
but the final link between morphisms of symmetric lists and permutation groups is only partially formalized in his work.
We contribute
the necessary adaptations and translations of Piceghello's work from the Homotopy Type Theory framework into the non-HoTT
framework of categories in Lean/\mathlib, and our main contribution here is the formalized identification of morphisms of
symmetric lists with permutations via labellings of their morphisms by elements of suitable Coxeter groups
(Theorem~\ref{w_faithful} and Corollary~\ref{equiv_type_indx}).

We also believe our work is of interest in explaining \emph{how} to precisely leverage the coherence theorem to obtain
Theorem~\ref{main_th}, and clarifying this process.
The general ideas are certainly known to the mathematical community,
but the actual details are more often than not left to the reader.

Our formalization is available online\footnote{\href{https://github.com/robin-carlier/SymmMonCoherence/tree/\refCommitRepo}{\texttt{https://github.com/robin-carlier/SymmMonCoherence/tree/\refCommitRepo}}}. The full repository is about 12.5kLoC, excluding comments. Our code compiles using the lean toolchain \texttt{leanprover/lean4:v4.28.0}, and depends on \mathlib's commit
\texttt{8f9d9cff6bd728b17a24e163c9402775d9e6a365}.

\subsubsection*{Outline}
In Section~\ref{coh-th-section}, we focus on the coherence theorem~\ref{coherence_theorem}. We first introduce the category
of symmetric lists in \ref{pres-slist}, adapting the work in \cite{Piceghello} to the setting of Lean/\mathlib. This category will play a central role in the rest of the paper as a convenient model of the free
symmetric monoidal category on a set. In~\ref{charac-morphisms} we study morphisms in the category of symmetric lists: we first briefly
recall and formalize the link between Coxeter groups of type \(\mathrm{A}_n\) and permutations in~\ref{cox-groups-subsubsection}
and make use of this material in~\ref{subsubsection-morph-slist-perm} to formalize that morphisms of symmetric lists are in
one-to-one correspondence with permutations by showing that they can be faithfully labeled as elements of a Coxeter group of type
\( \mathrm{A}_\infty \) (Theorem~\ref{w_faithful}).
In~\ref{slist-vs-fsmc}, we follow Piceghello's method to prove that symmetric lists are indeed free symmetric monoidal categories,
which completes the proof of Theorem~\ref{coherence_theorem}.

In Section~\ref{section-pseudo-out-of-pith}, we recall the basics of bicategories of spans and explain our formalization of them. Then,
we introduce in Definition~\ref{def-pbc} the precise structure needed to define pseudofunctors out of the
pith of the span bicategory of a category with pullbacks.

In Section~\ref{section-packaging} we construct the necessary data to build the pseudofunctor of Theorem~\ref{main_th} using
the machinery from Section~\ref{section-pseudo-out-of-pith}. In ~\ref{subsection-kleisli-bicat}, we first abstract the target bicategory from
\( \mathrm{Cat} \) to a suitable Kleisli bicategory \( \mathcal{K}_{\mathrm{SList}}^{\mathrm{op}} \) for the theory of symmetric monoidal categories and we
show that one can interpret a symmetric monoidal category as a \( \mathrm{Cat} \)-valued
pseudofunctor from the opposite of the Kleisli bicategory (Proposition~\ref{exists-pseudo-kleisli}).
In ~\ref{subsubsection-kleisi-multiset}, we explain how to leverage the coherence theorem~\ref{coherence_theorem}
to show that in good cases, computations in the Kleisli bicategory behave like matrix computations with multisets. Finally, we build a pseudofunctor \(\Lambda: \mathrm{Pith}\left(\mathrm{Span}\left({\mathrm{Fin}}\right)\right) \to
\mathcal{K}_{\mathrm{SList}}^{\mathrm{op}}\) in Proposition~\ref{pbc-system-propdef}, finalizing the construction of the pseudofunctor described in Theorem~\ref{main_th}.

\section{The coherence theorem}\label{coh-th-section}

\subsection{Symmetric lists: presenting a category via generators and relations}\label{pres-slist}

The coherence theorem for (non-symmetric) monoidal categories was already formalized in \mathlib
by Markus Himmel,
following a classical proof by normalization due to Beylin and Dybjer \cite{beylin1996} that was suited for type-theory
based proof assistants and originally formalized in the proof assistant ALF.

One of the key ideas of this proof is to make the statement of the theorem a statement about free monoidal categories on
types: such free monoidal categories are groupoids and have at most one morphism between any pair of objects.
By construction, the morphisms in the free monoidal category on a type \(\mathsf{C}\) are exactly those one can build
out of associators, left or right unitors and tensor products of other morphisms. The classical statement that
``every diagram commutes'' is thus encoded by this statement. To prove this property of morphisms in free monoidal
categories, one normalizes objects in the free monoidal category on a type \(T\) to essentially show that it
is equivalent to the discrete category on the set of lists.

The case of symmetric monoidal categories is harder as one cannot hope for the free symmetric monoidal category to have
at most one morphism between objects: there are symmetric monoidal categories \(\mathsf{C}\) where the identity
\(\beta_{c, c}^{\mathsf{C}} = \mathrm{Id}_{c \otimes c}\) does not hold, despite the morphisms having the same source and
target.
It remains reasonable to expect that free symmetric monoidal categories can have their
associative and unital parts normalized into a list-like structure (as one can do so after forgetting the braidings),
but the presence of braiding isomorphisms means that the resulting category of normal objects will not be discrete
and will yield a non-trivial category structure on lists.
Piceghello presented this category structure \cite[Def. 4.11]{Piceghello} as symmetric lists, a direct categorification of
multisets, which we describe here:
\begin{definition}{\cite[Def. 4.11]{Piceghello}}\label{def-slist-informal}
  Let \( J \) be a set. The category of \emph{symmetric lists} on \(J\), denoted \( \mathrm{SList}\left(J\right) \)
    is defined as the category presented by the following generators and relations:
  \begin{itemize}
    \item Objects of \(\mathrm{SList}\left(J\right)\) are lists of elements of \(J\).
    \item Morphisms are generated by the following (inductive) rules: \begin{enumerate}
      \item For all \( a, b \in J \) and \(l : \mathrm{SList}\left(J\right) \), there is a
        morphism \[ \mathrm{sw}_{a,b,l} : (a \dblcolon b \dblcolon l) \to (b \dblcolon a \dblcolon l). \]
      \item If \(f : a \to b \) is a morphism, there is a morphism
        \((x \dblcolon_m f) : (x \dblcolon a) \to (x \dblcolon b).\)
      \end{enumerate}
    \item Morphisms are subject to the following relations: \begin{enumerate}
        \item The constructions \(f \mapsto x \dblcolon_m f\) respect compositions and identities.
        \item The morphism \(\mathrm{sw}_{a,b,l}\) is natural in \(l\)
            (when interpreting \(x \mapsto (x \dblcolon -)\) as a functor using the previous relation).
        \item The symmetry relation
            \(\mathrm{sw}_{b,a,l}\mathrm{sw}_{a,b,l} = \mathrm{Id}_{a \dblcolon b \dblcolon l}\) holds.
        \item The diagram
        \[\begin{tikzcd}
                & {b \dblcolon a \dblcolon c \dblcolon l} && {b\dblcolon c \dblcolon a \dblcolon l} & \\
                {a\dblcolon b \dblcolon c \dblcolon l} &&&& {c\dblcolon b \dblcolon a \dblcolon l} \\
                & {a\dblcolon c \dblcolon b \dblcolon l} && {c \dblcolon a \dblcolon b \dblcolon l}
                \arrow["{b \dblcolon_m \mathrm{sw}_{a,c,l}}", from=1-2, to=1-4]
                \arrow["{\mathrm{sw}_{b, c, a\dblcolon l}}", from=1-4, to=2-5]
                \arrow["{\mathrm{sw}_{a, b, c\dblcolon l}}", from=2-1, to=1-2]
                \arrow["{a \dblcolon_m \mathrm{sw}_{b,c,l}}"', from=2-1, to=3-2]
                \arrow["{\mathrm{sw}_{a,c,b \dblcolon l}}"', from=3-2, to=3-4]
                \arrow["{c \dblcolon_m \mathrm{sw}_{a,b,l}}"', from=3-4, to=2-5]
        \end{tikzcd}\]
        commutes for every \(a, b, c \in J\) and \(l \in \mathrm{SList}(J)\).
      \end{enumerate}
  \end{itemize}
\end{definition}

In the setting of Homotopy Type Theory, in which \cite[Def. 4.11]{Piceghello} takes place,
this definition can be realized at once as a (1-truncated) higher inductive type.
In the setting of \mathlib, we cannot make such a definition directly, and instead have to
manually translate and interpret the 1-truncated HoTT definition in the model of groupoids provided by \mathlib.

Therefore, formalizing the definition above requires several intermediate steps, which together constitute
the standard construction of a category presented via generators and relations. In the infrastructure provided by
\mathlib, the steps take the following form
\begin{enumerate}
  \item Define a type \lean{V} of generators for the objects, and introduces a \lean{Quiver} instance on this type of generators, endowing it with generating arrows.
  \item Take the path category \lean{Paths V} on the quiver from the previous step, giving a free category where arrows are formal compositions (paths) of generating morphisms.
  \item Define a relation \lean{homRel} between morphisms in \lean{Paths V} corresponding to the relations in the final category.
    In \mathlib, the type of such relations is \lean{HomRel (Paths V)}.
  \item Take the quotient category of \lean{Paths V} by the relation \lean{homRel} from the previous point.
\end{enumerate}
We carry out this program in our case:
\begin{leancode}
inductive SListQuiv (C : Type u) where
  | nil : SListQuiv C
  | cons (head : C) (tail : SListQuiv C) : SListQuiv C

infixr:67 " ::… " => SListQuiv.cons

inductive Hom {C : Type u} : SListQuiv C → SListQuiv C → Type u
  | swap (x y : C) (l : SListQuiv C) :
      Hom (x ::… (y ::… l)) (y ::… (x ::… l))
  | cons (z : C) {l l' : SListQuiv C} :
      Hom l l' → Hom (z ::… l) (z ::… l')

instance : Quiver (SListQuiv C) where Hom := Hom
structure FreeSListQuiv where p : Paths (SListQuiv C)
instance : Category (FreeSListQuiv C) :=
  inferInstanceAs (Category <| InducedCategory _ FreeSListQuiv.p)
\end{leancode}
In the HoTT setting, the constructor \lean{Hom.cons} is implicit: the fact that it is a function
in HoTT automatically gives the constructor \lean{cons} an action on paths.
When interpreting this definition in a non-HoTT setting,
extra constructors have to be added, and it has to be extended manually to an
endofunctor on the paths category in order to represent its action on paths.
In the code block below, \lean{ι C} is the inclusion prefunctor from \lean{SListQuiv C}
to \lean{FreeSListQuiv C}, \lean{β₁_} is a notation for the generating swap morphism,
and \lean{::_} (resp. \lean{::_ₘ}) is the action on objects (resp. morphisms) of the extension to
\lean{FreeSListQuiv C} of the \lean{cons} constructor.
\begin{leancode}
inductive HomEquiv : HomRel (FreeSListQuiv C)
  | swap_naturality (X Y : C) {l l' : SListQuiv C} (f : l ⟶ l') :
      HomEquiv
        (β₁_ X Y ((ι C).obj l) ≫
          (Y ::_ₘ (X ::_ₘ ((ι C).map f))))
        ((X ::_ₘ (Y ::_ₘ ((ι C).map f))) ≫
          (β₁_ X Y ((ι C).obj l')))
  | swap_swap (X Y : C) (l : FreeSListQuiv C) :
      HomEquiv (β₁_ X Y l ≫ β₁_ Y X l) (𝟙 _)
  | triple (X Y Z : C) (l : FreeSListQuiv C) :
      HomEquiv
        (β₁_ X Y (Z ::_ l) ≫ (Y ::_ₘ (β₁_ X Z l)) ≫ β₁_ Y Z (X ::_ l))
        ((X ::_ₘ (β₁_ Y Z l)) ≫ β₁_ X Z (Y ::_ l) ≫ Z ::_ₘ (β₁_ X Y l))
  | cons (X : C) {l l' : FreeSListQuiv C} (f f' : l ⟶ l') :
      HomEquiv f f' → HomEquiv (X ::_ₘ f) (X ::_ₘ f')

def SList := CategoryTheory.Quotient (FreeSListQuiv.HomEquiv C)
  deriving Category
\end{leancode}
Again, due to the non-HoTT nature of \mathlib's framework, extra relations have to be added in order to make
the definition sound: \lean{HomEquiv.swap_naturality} and \lean{HomEquiv.cons} are not explicitly present in the
HoTT presentation, as higher inductive types in HoTT ensure that all constructors are natural and functorial with
respect to paths.
\begin{remark}
  Besides the adaptations from the HoTT setting to the interpretation in the non-HoTT setting of Lean, there are
  two minor differences between our construction and Piceghello's.
  \begin{enumerate}
    \item If we were to strictly interpret everything in the ``groupoid model'' provided by \mathlib's groupoids,
      we would need to replace the free category on \lean{SListQuiv}
      by the free groupoid on \lean{SListQuiv} to ensure that we stay within the realm of
      groupoids. This would add extra generators for formal inverses of the generating morphisms, as well as extra relations
      ensuring that the formal inverses define inverses in the quotient category.
      In the case of symmetric lists, it is easy to show that the category resulting from our definitions is a
      groupoid, and so we prefer the more direct construction in the setting of categories.
    \item Our definition is less general:
      if we were to strictly interpret the construction in \cite{Piceghello} in the groupoid
      model, we would be building symmetric lists on a groupoid,
      rather than symmetric lists on a set (which would correspond to the case of a discrete groupoid).
      We would then need extra relations ensuring naturality of the generating morphisms with respect to every parameter
      in \( C \), making the presentation of the category and the study of its morphisms more complex.
      For the purposes of stating our version of the coherence theorem and unbiasing symmetric monoidal categories,
      this extra generality is not needed, so we do not implement it.
  \end{enumerate}
\end{remark}
A version of the universal property of the category of symmetric lists as a category presented by generators and relations can
be spelled out explicitly as the following:
\begin{lemma}\label{up-slist}
  Let \(\mathsf{D}\) be a category.
  \begin{enumerate}[label=(\roman*)]
    \item
        Suppose given the following data:
        \begin{itemize}
          \item An object \(F_{\mathrm{nil}} \in \mathsf{D}\).
          \item For every \(c \in C\), an endofunctor \(\gamma^{F}_{c} : \mathsf{D} \to \mathsf{D}\).
          \item For all \(a, b \in C\), a natural isomorphism
            \[\tau^{F}_{a,b} : \gamma^{F}_a \gamma^{F}_b \iso \gamma^F_{b}\gamma^{F}_a.\]
        \end{itemize}
        Assume the data satisfies the following conditions:
        \begin{itemize}
          \item For all \(a, b \in C\), the equality \(\tau^{F}_{a,b}\tau^{F}_{b,a} = \mathrm{Id}\) holds.
          \item For all \(a, b, c \in C\), the diagram
            \[\begin{tikzcd}[cramped]
              & {\gamma^{F}_b\gamma^F_a\gamma^F_c} & {\gamma^{F}_b\gamma^F_c\gamma^F_a} & \\
              {\gamma^{F}_a\gamma^F_b\gamma^F_c} &&& {\gamma^{F}_c\gamma^F_b\gamma^F_a} \\
              & {\gamma^F_a\gamma^F_c\gamma^F_b} & {\gamma^F_c\gamma^F_a\gamma^F_b}
              \arrow["{\gamma^F_b \cdot \tau_{a,c}^F}", from=1-2, to=1-3]
              \arrow["{\tau_{b,c}^F \cdot \gamma^F_a}", from=1-3, to=2-4]
              \arrow["{\tau_{a,b}^F \cdot \gamma^F_c}", from=2-1, to=1-2]
              \arrow["{\gamma^F_a \cdot \tau_{b,c}^F}"', from=2-1, to=3-2]
              \arrow["{\tau_{a,c}^F\cdot\gamma^F_b}"', from=3-2, to=3-3]
              \arrow["{\gamma_c^F\cdot\tau^F_{a,b}}"', from=3-3, to=2-4]
            \end{tikzcd}\]
            commutes.
        \end{itemize}
        This data defines a functor \(F : \mathrm{SList}(C) \to \mathsf{D}\) equipped with isomorphisms \begin{align*}
          \begin{array}{lcl} \upsilon_{\mathrm{nil},F} :&  F([\ ]) &\simeq F_{\mathrm{nil}},\\
          \upsilon_{\mathrm{cons}, F, c} :& F \circ (c \dblcolon \cdot ) &\simeq \gamma^F_c \circ F,\end{array}
        \end{align*}
        and such that
        \(F(\mathrm{sw}_{a,b,l})\) is identified with \(\left(\tau^F_{a,b}\right)_{F(l)}\) through these isomorphisms.
    \item Let \(F, G : \mathrm{SList}(C) \to \mathsf{D}\) be functors. Suppose we are given the following data:
      \begin{itemize}
        \item A morphism \(\phi_{\mathrm{nil}} : F([\ ]) \to G([\ ])\).
        \item For every \(c \in C\), \(l \in \mathrm{SList}(C)\) and morphism \(\mathrm{ind} : F(l) \to G(l) \), a morphism
          \(\phi_{c,l}(\mathrm{ind}) : F(c\dblcolon l) \to G(c\dblcolon l)\).
      \end{itemize}
      Assume further that the following conditions hold:
      \begin{itemize}
        \item For all \(a, b \in C\), for all \(l \in \mathrm{SList}(C)\) and for all morphisms \(\mathrm{ind} : F(l) \to G(l)\), the diagram
          \[\begin{tikzcd}[cramped]
            {F(a\dblcolon b\dblcolon l)} && {F(b\dblcolon a\dblcolon l)} \\
            {G(a\dblcolon b\dblcolon l)} && {G(b\dblcolon a\dblcolon l)}
            \arrow["{F(\mathrm{sw}_{a,b,l})}", from=1-1, to=1-3]
            \arrow["{\phi_{a,b\dblcolon l}(\phi_{b, l}(\mathrm{ind}))}"', from=1-1, to=2-1]
            \arrow["{\phi_{b,a\dblcolon l}(\phi_{a, l}(\mathrm{ind}))}", from=1-3, to=2-3]
            \arrow["{G(\mathrm{sw}_{a,b,l})}"', from=2-1, to=2-3]
          \end{tikzcd}\]
        commutes.
        \item Given \(c \in C\), a morphism \(f : l \to l'\) in \(\mathrm{SList}(C)\) and morphisms
          \(\mathrm{ind}_l : F(l) \to G(l)\),
          \(\mathrm{ind}_{l'} : F(l') \to G(l')\) satisfying \(\mathrm{ind}_{l'} \circ F(f) = G(f) \circ \mathrm{ind}_{l}\),
          the diagram
          \[\begin{tikzcd}[cramped]
                  {F(c\dblcolon l)} && {F(c\dblcolon l')} \\
                  {G(c\dblcolon l)} && {G(c\dblcolon l')}
                  \arrow["{F(c\dblcolon_m f)}", from=1-1, to=1-3]
                  \arrow["{\phi_{c,l}(\mathrm{ind}_l)}"', from=1-1, to=2-1]
                  \arrow["{\phi_{c,l'}(\mathrm{ind}_{l'})}", from=1-3, to=2-3]
                  \arrow["{G(c\dblcolon_m f)}"', from=2-1, to=2-3]
          \end{tikzcd}\]
          commutes.
      \end{itemize}
      The data defines a unique natural transformation \(\phi : F \to G\) such that
        \(\phi_{[\ ]} = \phi_{\mathrm{nil}}\) and such that \(\phi_{c\dblcolon l} = \phi_{c,l}(\phi_l)\).
  \end{enumerate}
\end{lemma}

We implement the first point of the above lemma using a dedicated structure \\
\smcLinkRange{SymmMonCoherence/SList/Basic.lean}{1130}{1151}{\lean{RecursiveFunctorData}} encapsulating the data defined in the first point, giving rise to a declaration
\smcLinkRange{SymmMonCoherence/SList/Basic.lean}{1213}{1231}{\lean{RecursiveFunctorData.functor}} providing the functor attached to the data.
The second point is implemented as a declaration \smcLinkRange{SymmMonCoherence/SList/Basic.lean}{1130}{1151}{\lean{recNatTrans}} taking
directly the necessary data as parameter.

\begin{remark}\label{eq_objects}
  Our statement of Lemma~\ref{up-slist} is bicategorical in the sense that it characterizes functors out of the category of symmetric lists up to a unique isomorphism. In fact, the data in the first point defines functors \emph{uniquely}, and the isomorphisms
  \(\upsilon_{\mathrm{nil}}\) and  \(\upsilon_{\mathrm{cons}}\) that characterize
the resulting functors can be (componentwise) definitional equalities. We intentionally avoid stating things this way, for reasons we explain
below.

A recurring theme when formalizing category theory in the setting of a dependent type theory like Lean is
  that, generally speaking, equalities of objects of categories should be avoided when possible.
One of the reasons is that types of morphisms, and hence functions like composition of morphisms, actively depend
on objects. Given a category \lean{C} and terms \lean{x, y, z : C}, in the presence of an equality \lean{h : x = y} in context,
a morphism \lean{f : x ⟶ z} will not directly type check as a morphism \lean{y ⟶ z} and a casting operation has to be performed on \lean{f}.
When the equality \lean{h} is an equality of free variables in contexts, performing a cast is easily done.
If \lean{h} is a more complex expression, direct substitution is usually not possible, and, while it is \emph{theoretically} possible
to use the induction principle on equalities with carefully crafted induction motives to perform substitutions, it is in practice extremely
tedious to do so, especially since this has to be done repeatedly in every proof where the situation arises.
Situations like these are colloquially referred to as ``DTT Hell'' within the Lean community.

To make the situation with equality of objects slightly more manageable, \mathlib made the choice to provide casts only for identity morphisms, these are
called \lean{eqToHom}: given \lean{h : x = y}, \lean{eqToHom h : x ⟶ y} is the equality \lean{h} as a morphism. Usage of \lean{eqToHom} as
a correction term when composing morphisms whose sources and targets only match up to propositional equality is a standard method in \mathlib,
but it only partially alleviates the inherent trouble of working with equalities of objects and it is still considered better practice
to try to not end up in situations where these are needed in the first place.

When an equality of objects \lean{h : x = y} is \emph{definitional}, the situation is better and the type checker will accept
\lean{f : x ⟶ z} as a valid morphism \lean{y ⟶ z}, but there remains an important technical subtlety:
in Lean, for performance reasons, checks for definitional equalities depend on a setting called \emph{transparency}, which
controls how definitions are allowed to be unfolded when checking for definitional equality of terms.
Most of the automation tactics in lean (\lean{simp}, \lean{grind}, \lean{rw}, etc.,) work at the ``reducible'' transparency level, where
most definitions are \emph{not} unfolded. If a definitional equality \lean{h : c = c'} does not hold at reducible
transparency (for instance, if it is gated behind a \lean{def}), tactics may be unable to use terms that depends on it. Over-reliance
on non-reducible definitional equalities is often colloquially referred to as ``defeq abuse'' within the Lean community.

In our situation, the definitional equalities involving \lean{RecursiveFunctorData.functor} may hold, but not at reducible transparency and
relying on them could cause problems when trying to automate proofs and computations involving such recursively-defined functors.
Hence, we take the opinionated route of considering that these definitional equalities are an
implementation detail, and reflect this in our statement of Lemma~\ref{up-slist}.

In practice, in our implementation, we leverage the Lean module system introduced in
\texttt{v4.26.0} to make it so that the body of \lean{RecursiveFunctorData.functor} is not ``exposed'': this ensures that
no declaration in the ``public'' scope can use the definitional equalities that underlie \(\upsilon_{\mathrm{nil}}\) and the components of \(\upsilon_{\mathrm{cons}}\).
We still export lemmas about existence of equalities in the public scope, as they can provide a convenient shortcut to show that some diagrams made exclusively
of morphisms of the form \(\upsilon_{\mathrm{nil}}\) and  \(\upsilon_{\mathrm{cons}}\) commute,
but their definitional nature is sealed away, making it impossible to abuse.
\end{remark}

\subsection{Studying morphisms of symmetric lists}\label{charac-morphisms}

The key point in Mac Lane’s original proof of the coherence theorem is the idea that the hexagon and symmetry relations
correspond to the relations that present the symmetric group on \(n + 1\) letters \(\Sigma_{n + 1}\) as a
Coxeter group of type \(\mathrm{A}_n\): this presentation is the isomorphism
\[
  \left\langle \left(s_{i}\right)_{1 \le i \le n}\quad \left|
    \begin{array}{clcl}
      s_i^2 &=& e & \\
      s_is_{i+1}s_i &=& s_{i+1}s_is_{i+1} & \\
      s_i s_j &=& s_js_i & \mathrm{if}\ |i - j| > 1
    \end{array}\right. \right\rangle \cong \Sigma_{n + 1}
\]
that sends the generator \(s_i\) to the permutation \(\left(i\ i + 1\right)\) (once the set of letters has been
enumerated). One can indeed see a clear link
between this presentation and our categories of symmetric lists: given a symmetric list
\(\left[x_0, \ldots, x_n\right]\), we could label \(s_0\) a morphism of the form \[
  \mathrm{sw}_{x, y, \left[x_2, \ldots, x_n\right]} : \left[x, y, \ldots, x_n\right] \to
  \left[y, x, \ldots, x_n \right] \]
which conveniently happens to realize the permutation \((0\ 1)\) on the indices of the list, and the symmetry
relation would then read \(s_0^2 = 1\).
We could also label \(s_1\) any morphism of the form
\(x \dblcolon \mathrm{sw}_{y, z, \left[x_3, \ldots, x_n \right]} \), and the hexagon relation would then conveniently
read as \( s_0 s_1 s_0 = s_1 s_0 s_1 \). Finally, labelling e.g., \(s_2\) a morphism of the form
\(x \dblcolon y \dblcolon \mathrm{sw}_{z, t, \left[x_4, \ldots, x_n \right]}\), the relation on morphisms of
symmetric lists that asserts naturality of the swap would then read as the equation \(s_0s_2 = s_2s_0\).

The procedure described above provides the core idea for establishing a correspondence between morphisms of
symmetric lists and permutations. However, this idea must be refined before it becomes formal enough for a proof assistant to accept.

\subsubsection{Coxeter groups of type \( \mathrm{A}_n \) and permutation groups}\label{cox-groups-subsubsection}

Before even attempting to formalize the above procedure, the link between Coxeter groups of type \(\mathrm{A}_n\) and
permutations of the set \(\{0,\ldots, n\}\) needs to be formalized. The library \mathlib
defines a \mathlibLinkRange{Mathlib/GroupTheory/Coxeter/Basic.lean}{155}{160}{Coxeter system} as the structure of an
isomorphism between a group and the group presented by generators and relations from a Coxeter matrix.
For \mathlib, a \mathlibLinkRange{Mathlib/GroupTheory/Coxeter/Matrix.lean}{66}{74}{Coxeter matrix} is a
(possibly infinite) square matrix \(M\) on a set \( I \) with positive integer coefficients, that is symmetric,
such that all diagonal coefficients are one and such that all off-diagonal coefficients are not equal to one.

Every Coxeter matrix \( M \) on a set \( I \) defines a group \( \overline{M} \) by quotienting the
free group on \( I \) by the \emph{Coxeter relations}
\(\left((s_i s_j)^{M_{ij}} = e\right)_{i,j \in I^2}\),
where \(s_i\) is the generator corresponding to \(i \in I\).

In this work, we will mainly be working with the Coxeter matrix \(\mathrm{A}_n\) on the set \(\{0,\cdots,n-1\}\), represented by
\[
\mathrm{A}_n = \begin{pmatrix}
1 & 3 & 2 & \cdots & 2 \\
3 & 1 & 3 & \ddots & \vdots \\
2 & 3 & 1 & \ddots & 2 \\
\vdots & \ddots & \ddots & \ddots & 3 \\
2 & \cdots & 2 & 3 & 1
\end{pmatrix}
\]
as well as with its infinite variant \(\mathrm{A}_\infty\) (as a Coxeter matrix on the set \(\mathbb{N}\)).

At the time of writing, despite some theory on abstract Coxeter systems, \mathlib does not produce any non-trivial term of type
\lean{CoxeterSystem G} (non-trivial meaning here that \( G \) is not definitionally the group presented by a Coxeter
matrix).
Following the textbook proof of \cite[\S 1.5]{bjorner2005}, we formalize a (classical) criterion to determine
when a given set of degree 2 generators in a group \( G \)
satisfying correct relations extends to a presentation of \( G \) as a Coxeter group: we
introduce a \smcLinkRange{SymmMonCoherence/CoxeterGroupRecognition.lean}{38}{43}{structure}
\begin{leancode}
/- Below, M.simple i is the element of the Coxeter group attached to M that
corresponds to the generator `i`. -/
structure PreCoxeterSystem {B : Type*} (M : CoxeterMatrix B) (G : Type*) [Group G] where
  hom : M.Group →* G
  surjective_hom : Function.Surjective hom
  orderOf_eq (i j : B) : M i j = orderOf (hom (M.simple i) * hom (M.simple j))
  hom_simple_ne_one (i : B) :  hom (M.simple i) ≠ 1
\end{leancode}
encoding a set of degree \( 2 \) generators in \( G \) satisfying suitable equations as a surjective group homomorphism
from the corresponding Coxeter group to \( G \).
The criterion is then called the \emph{exchange property}~\cite[p. 18]{bjorner2005}, which we explain with Lean
code in the listing below:
\begin{leancode}
variable {B : Type*} (M : CoxeterMatrix B) (G : Type*) [Group G] (S : PreCoxeterSystem M G)
/-- The length of an element `g` is the minimal length (in the sense of Coxeter systems)
of the preimages of `g` in `M.Group`. This corresponds to the minimal length of
a word of generators needed to express `g`. -/
noncomputable def length (g : G) : ℕ :=
  Nat.find <| show (M.toCoxeterSystem.length '' (S.hom⁻¹' {g})).Nonempty by …

local prefix:100 "ℓ " => S.length

/-- For ω a list of generators, π ω is the product in the Coxeter group attached to M
of the word ω. -/
local prefix:100 "π " => M.toCoxeterSystem.wordProd

/-- For ω a list of generators, φ ω is the product in G of the generators, via
the morphism S.hom :  -/
local notation "φ " x:max => S.hom (π x)

/-- A word is reduced if its length in the Coxeter group M.Group
(which is the minimal number of generators required to express it)
is equal to its length in G. -/
abbrev IsReduced (ω : List B) : Prop := ℓ (S.hom <| π ω ) = ω.length

/-- The "exchange property" for a pre-Coxeter system `S`: if a word of generators
`w = s₁⋯sₖ` in `G` is reduced and reduces further when multiplying on the left by a
generator `s`, then there exists `1 ≤ i ≤ k` such that
`sw = s₁⋯ŝᵢ⋯sₖ`. -/
def ExchangeProperty : Prop :=
  ∀ (ω : List B) (_ : S.IsReduced ω) (b : B) (_ : ℓ (φ (b :: ω)) ≤ ω.length),
    ∃ (i : ℕ) (_ : i < ω.length), φ (b :: ω) = φ (ω.eraseIdx i)
\end{leancode}
Following the exact same proof as in \cite[Thm. 1.5.1]{bjorner2005},
we can \smcLink{SymmMonCoherence/CoxeterGroupRecognition.lean}{763}{formalize} that the
group homomorphism \lean{S.hom} of a pre-Coxeter system \lean{S} that satisfies the exchange
property is injective. Our formalization follows closely loc. cit. and we refer the reader to our implementation for more details.

The criterion then applies to symmetric groups: we can define a pre-Coxeter system
\smcLinkRange{SymmMonCoherence/SymmetricGroupCoxeterSystem.lean}{81}{103}{\lean{Fin.preCoxeterSystem n}} on the group
\(\Sigma_{n + 1}\) for the matrix \(\mathrm{A}_n\) by sending the generator \( i \) to the permutation \((i\ i + 1)\)\footnote{
  In the pull request \mathlibPR{35218} to \mathlib, independent of our work, Kim Morrison defined that same morphism and showed its
surjectivity as well.}.
Following again the material from \cite[Thm. 1.5.1]{bjorner2005}, we can
\smcLinkRange{SymmMonCoherence/SymmetricGroupCoxeterSystem.lean}{277}{367}{formalize}
that \smcLinkRange{SymmMonCoherence/SymmetricGroupCoxeterSystem.lean}{81}{103}{\lean{Fin.preCoxeterSystem n}}
satisfies the exchange property for all \(n\). This involves relating the notion of length for this pre-Coxeter system with the
inversion count of a permutation.

Finally, we record a lemma that we will make use of in our applications
\begin{lemma}\label{group-vs-monoid}
  Let \( M \) be a Coxeter Matrix on a set \(I\), the group \(\overline{M}\) is also presented as a monoid by the Coxeter relations, i.e.,
  the kernel of the monoid homomorphism \(\mathrm{FreeMonoid}(I) \to \overline{M}\) is the smallest multiplicative
  congruence on \(\mathrm{FreeMonoid}(I)\) containing the words \(\left((s_i s_j)^{M_{ij}}\right)_{i, j}\).
\end{lemma}
The proof of this lemma reduces to the fact that \(s_i^{-1} = s_i\) in \( \overline{M} \) for every \(i\).

\subsubsection{Morphisms of symmetric lists and permutations}\label{subsubsection-morph-slist-perm}

We are now ready for the study of morphisms of symmetric lists.
One of the first things to notice is that the way we label morphisms forgets their source and target.
We abstract the idea of multiplicatively labelling morphisms of a category by elements of a monoid:
\begin{definition}\label{weight-def}
  Let \( \mathsf{C} \) be a category and \( M \) be a monoid. A \emph{\( M \)-weight} on \(\mathsf{C}\) is a functor
  from \(\mathsf{C}\) to the opposite\footnote{In \mathlib, compositions of morphisms in a category are reversed compared
    to the usual way it is written in pen-and-paper mathematics, i.e., it ``follows the arrows'';
    this is why an opposite is taken here.}
  of the category \(\mathrm{B}M\) that has a single object and the monoid \( M \) as endomorphisms of that single object.
\end{definition}
This definition provides an interface between morphisms in a category and elements of a monoid
in a way that turns composition into multiplication and that sends identities to the unit of the monoid.

The labeling of morphisms that we described in the opening paragraph of this section should be thought of as a
weight on the category of symmetric lists valued on a Coxeter group.
In fact, since we want to relate the relations that define symmetric lists with the relations that present
permutation groups as Coxeter groups, it is more convenient to first define weights at the level of free objects.
\begin{propdef}
  The inductive assignment \(w_0\) on arrows of \(\mathrm{SListQuiv}(C)\) characterized by the formulas
  \begin{align*}
    w_0(\mathrm{sw}_{a, b, l}) &= 0\\
    w_0(x \dblcolon_m f) &= 1 + w_0(f)
  \end{align*}
  extends to a weight \( w_0 \) on \(\mathrm{FreeSListQuiv}(C)\) with values in the free monoid on \(\mathbb{N}\).
  The weight \( w_0 \) descends to a weight \(w\) on \(\mathrm{SList}(C)\) with values in the Coxeter group
  \(\overline{\mathrm{A}_{\infty}}\) in a way such that the diagram
  \[\begin{tikzcd}
          {\mathrm{FreeSListQuiv}(C)} & {\left(\mathrm{B}\mathrm{FreeMonoid}(\mathbb{N})\right)^{\mathrm{op}}} \\
          {\mathrm{SList}(C)} & {\left(\mathrm{B}\overline{\mathrm{A}_{\infty}}\right)^{\mathrm{op}}}
          \arrow["{w_0}", from=1-1, to=1-2]
          \arrow["{\pi_{C}}"', from=1-1, to=2-1]
          \arrow[from=1-2, to=2-2]
          \arrow["w"', from=2-1, to=2-2]
  \end{tikzcd}\]
  commutes.
\end{propdef}
Informally, the weight \(w_0\) records the indices at which the sequence of swaps corresponding to a morphism in \(\mathrm{FreeSListQuiv}(C)\)
is happening, and the weight \(w\) realizes a morphism of symmetric list as an element of \(\overline{A_\infty}\).
The latter can also be thought of as a permutation group: the Coxeter group \(\overline{\mathrm{A}_\infty}\) admits a
group homomorphism to the group of permutations of \(\mathbb{N}\) by sending the generator \(s_k\) to the permutation \((k\ k + 1)\).
One can use the identification of \(\overline{\mathrm{A}_n}\) with \(\Sigma_{n + 1}\) for every \(n\) and the fact that
\(\overline{\mathrm{A}_\infty}\) is a filtered colimit of the groups \(\overline{\mathrm{A}_n}\) to show that this homomorphism is
\smcLinkRange{SymmMonCoherence/SymmetricGroupCoxeterSystem.lean}{646}{660}{injective}.

We can prove that, through this weight and the interpretation of elements of \(\overline{\mathrm{A}_\infty}\) as permutations,
morphisms of symmetric lists permute elements of the source and target lists in the expected way
\begin{leancode}
def toPerm : weight (SList C) (Equiv.Perm ℕ) := …

lemma SList.toPerm_app_lt_of_lt {L₁ L₂ : SList C} (f : L₁ ⟶ L₂) (k : ℕ) (hk : k < L₂.length) :
    (toPerm.app f) k < L₁.length := …

theorem SList.getElem_toList_toPerm {L₁ L₂ : SList C} (f : L₁ ⟶ L₂) (i : ℕ) (hi : i < L₂.length) :
    L₂.toList[i] = L₁.toList[toPerm.app f i]'(toPerm_app_lt_of_lt f i hi) := …
\end{leancode}
An induction argument on morphisms of the category \(\mathrm{FreeSListQuiv}\left(C\right)\) further shows that \(w_0\) is
faithful (\smcLinkRange{SymmMonCoherence/SList/Relations.lean}{484}{504}{\lean{eq_of_w₀_app_eq}}) as a functor, i.e., that morphisms in \(\mathrm{FreeSListQuiv}\left(C\right)\) with equal
labels are equal.
Now, part of the main theorem takes the following form:
\begin{theorem}\label{w_faithful}
  The functor \( w \) is faithful, that is, two morphisms of symmetric lists are equal if and only if they have the same label
  in \(\overline{\mathrm{A}_\infty}\).
\end{theorem}
In our formalization, this theorem is called \smcLinkRange{SymmMonCoherence/SList/Relations.lean}{1267}{1272}{\lean{injective_toAinf_app}}, which is itself a
direct consequence of the declaration \lean{SList.map_eq_of_w₂_eq} in the listing below and
which is a version of the theorem once morphisms have been appropriately
lifted to \lean{FreeSListQuiv C}.
\begin{leancode}
/- Here, w₂ is the A∞-valued weight on FreeSListQuiv C induced by w₀
and (π C) is the quotient functor FreeSListQuiv C ⥤ SList C. -/
theorem SList.map_eq_of_w₂_eq {L L' : FreeSListQuiv C} (f g : L ⟶ L')
    (h : w₂.app f = w₂.app g) :
    (π C).map f = (π C).map g := …
\end{leancode}
We will prove \lean{SList.map_eq_of_w₂_eq}.
\begin{proof}
In this proof, we will be using the same notations as in the Lean code listing above: \(L\) and \(L'\) will be lists,
seen as elements of the free category on the quiver generating morphisms of symmetric lists, and \(f, g\) will be
morphisms in this category. We will let \(\mathrm{toAinf}\) be the canonical projection
\(\mathrm{FreeMonoid}(\mathbb{N}) \to \overline{\mathrm{A}_\infty}\).
A first slightly technical step is to cast the hypothesis \lean{h} as an equality in a smaller monoid. We are
using here the fact from Lemma~\ref{group-vs-monoid} that Coxeter groups are
presented as monoids with the same generators and relations as their presentations as groups.
Recall that \lean{w₂.app f} is \lean{toAinf (w₀.app f)}.
By induction, one sees that the lengths of the underlying lists of objects are preserved along morphisms in
\lean{SListQuiv C} (and hence in \lean{FreeSListQuiv C} and \lean{SList C}) and that the formal word \(w_0(f)\) attached
to a morphism between symmetric lists of length \( n \) cannot involve labels greater than or equal to \(n - 1\).
Thus, in our situation, the words \(w_0(f)\) and \(w_0(g)\) lift through the injective morphism
\[\iota_{n - 1} : \mathrm{FreeMonoid}\left(\mathrm{Fin}(n-1)\right) \hookrightarrow
  \mathrm{FreeMonoid}\left(\mathbb{N}\right)\]
induced by the injection \(\mathrm{Fin}(n-1) \hookrightarrow \mathbb{N}\),
where \(n\) is the length of the list underlying the source object \(L\).
There is furthermore a commutative diagram
\[\begin{tikzcd}
	{\mathrm{FreeMonoid}(\mathrm{Fin}(n))} & {\mathrm{FreeMonoid}(\mathbb{N})} \\
	{\overline{\mathrm{A}_n}} & {\overline{\mathrm{A}_{\infty}}}
	\arrow["{\iota_n}", hook, from=1-1, to=1-2]
	\arrow["{\mathrm{toA_n}}"', two heads, from=1-1, to=2-1]
	\arrow["{\mathrm{toAinf}}", two heads, from=1-2, to=2-2]
	\arrow[from=2-1, to=2-2]
\end{tikzcd}\]
and the bottom horizontal map is injective, which we can see via the explicit equivalence between
\(\overline{\mathrm{A}_n}\) and \( \Sigma_{n+1} \), and the fact that \(\overline{\mathrm{A}_\infty}\) is a filtered colimit of
the groups \(\overline{\mathrm{A}_n}\), compatible with these injections.

Hence, the equality \lean{h} can be refined as an equality of words in the generators of the group
\(\overline{\mathrm{A}_{n - 1}}\).
In fact, we can assume that the words in
\(\mathrm{FreeMonoid}(\mathrm{Fin}(n - 1))\) that lift \(w_0(f)\) and \(w_0(g)\) are values on \(f\) and \(g\) of some
weight \(w'_{n - 1}\) on \(\mathrm{FreeSListQuiv(C)}\) with values in \(\mathrm{FreeMonoid}(\mathrm{Fin}(n - 1))\).
A weight \(w'_{n - 1}\) with this requirement can be constructed out of \(w_0\) by extending to the free monoids any
retraction of the injection
\(\mathrm{Fin}(n-1) \hookrightarrow \mathbb{N}\). In our
implementation, we use the function \(k \mapsto k\ \% (n - 1)\) as such a retraction, and the resulting weight is
named \lean{w₀Fin}.

Hence, the equality \(\mathrm{toAinf}(w_0(f)) = \mathrm{toAinf}(w_0(g))\) is equivalent to the equality
\begin{align}
  \mathrm{toA_{n - 1}}(w_{n - 1}'(f)) = \mathrm{toA_{n - 1}}(w_{n - 1}'(g)) \label{eq_h}
\end{align}
and we can use the characterizing property of a quotient monoid: this equality means that
\(w_{n-1}'(f)\) and \(w_{n-1}'(g)\) are words in the free monoid on \(\mathrm{Fin}\left( n-1 \right)\) that are
related by the minimal congruence on that monoid generated by the relations defining the Coxeter group
\(\overline{\mathrm{A}_{n - 1}}\).
In \mathlib, the minimal congruence generated by a relation is set up inductively, and we recall its definition
to fix the notations
\begin{leancode}
inductive ConGen.Rel [Mul M] (r : M → M → Prop) : M → M → Prop
  | of : ∀ x y, r x y → ConGen.Rel r x y
  | refl : ∀ x, ConGen.Rel r x x
  | symm : ∀ {x y}, ConGen.Rel r x y → ConGen.Rel r y x
  | trans : ∀ {x y z}, ConGen.Rel r x y → ConGen.Rel r y z → ConGen.Rel r x z
  | mul : ∀ {w x y z}, ConGen.Rel r w x → ConGen.Rel r y z → ConGen.Rel r (w * y) (x * z)
\end{leancode}
In our formalization, for Coxeter monoids, the relation taken on words is also inductively generated
\begin{leancode}
variable {B : Type*} (M : CoxeterMatrix B) in
inductive CoxeterMatrix.monoidRelations : FreeMonoid B → FreeMonoid B → Prop
  | intro (i j : B) : monoidRelations ((.of i * .of j) ^ M i j) 1
\end{leancode}
The proof proceeds by induction on the equality~\eqref{eq_h} using the recursor characterizing
\lean{ConGen.Rel} above, generalizing both lists \(L\) and \(L'\) in the process
so that inductive hypotheses that may appear in some cases can be applied to morphisms with possibly
different sources and targets.
The cases \lean{ConGen.refl} and \lean{ConGen.symm} pose no technical difficulties.
The cases \lean{ConGen.trans} and \lean{ConGen.mul} are slightly more involved, and make clearer why we decided to
first reduce to the case of the group \(\overline{\mathrm{A}_{n-1}}\): as we are performing an induction on
\emph{words} instead of \emph{morphisms} here, the inductive hypotheses
in these cases will in fact involve extra words in \(\mathrm{FreeMonoid}\left(\mathrm{Fin}\left(n - 1\right)\right)\)
that \textit{a priori} might not be attached to any morphism in \(\mathrm{FreeSListQuiv}\left(C\right)\). For instance,
the \lean{ConGen.trans} case will provide us with an extra word \(z\) related to both \(\mathrm{toA_{n - 1}}(w_{n - 1}'(f))\)
and \(\mathrm{toA_{n - 1}}(w_{n - 1}'(g))\), and the inductive hypothesis can only be applied if we can find an actual
morphism \(h\) to or from \(L\) such that \(\mathrm{toA_{n - 1}}(w_{n - 1}'(h)) = z\).
Fortunately, this is the case, and we can prove
\begin{leancode}
-- Here, `w₀Fin n` is the weight called w'ₙ₋₁ in the text above.
lemma exists_hom_of_weight_eq (i : FreeSListQuiv C) (n : ℕ) (hj : i.length = n + 2)
    (w : FreeMonoid (Fin (n + 1))) :
    ∃ j : FreeSListQuiv C, ∃ f : i ⟶ j, (w₀Fin n).app f = w := …
\end{leancode}
Note that the object \(j\) constructed as above is in fact unique as we have observed previously
(through \lean{SList.getElem_toList_toPerm})
that the target of a morphism is constrained by its source and the underlying permutation of the morphism.
If we were to work in \(\overline{\mathrm{A}_\infty}\) instead of \(\overline{\mathrm{A}_{n - 1}}\), such morphisms cannot exist if
the word \(z\) involves letters that are greater than or equal to the length of the sources and targets of \(f\) and \(g\).
The case \lean{ConGen.mul} is similar, as extra words appear in the inductive hypothesis;
we refer the reader to our \smcLinkRange{SymmMonCoherence/SList/Relations.lean}{1161}{1262}{formalization} for more details.

Finally, the case \lean{ConGen.of} follows the same pattern for all possible relations:
in this case, the words for \(\mathrm{toA_{n - 1}}(w_{n - 1}'(f))\) and \(\mathrm{toA_{n - 1}}(w_{n - 1}'(g))\) will
be explicit small words of generators, either of the form \(1\), \(s_i^2\),
\((s_i s_{i+1})^3\) or \((s_i s_j)^2\).
In each of these cases, we can explicitly build a morphism (using our
constructors) starting from \(L\), with the expected label under \(w'_{n - 1}\).
Since we already know that at the level of free monoids and
\(\mathrm{FreeSListQuiv}\left(C\right)\), the labels determine morphisms, these new morphisms built by hand must in fact
be equal to \(f\) and \(g\) and we can then observe that these second morphisms are equal as morphisms of
symmetric lists via computations.
As hinted in the opening paragraph of this subsection, the relation \(s_i^2 = 1\) reduces to the relation
\lean{swap_swap}, the relation \((s_is_{i+1})^3 = 1\) is essentially the relation \lean{hexagon}
and the relation \((s_i s_j)^2 = 1\) (for \(|i - j| > 1\)) corresponds to \lean{swap_naturality}.
\end{proof}

Thanks to the fact that \(\overline{A}_\infty\) injects into \(\mathrm{Perm}\left(\mathbb{N}\right)\) as permutations
with finite support, we can fully interpret morphisms in \(\mathrm{SList}\left(C\right)\) as actual permutations.
\begin{leancode}
def toEquiv {x y : SList C} (f : x ⟶ y) :
    Fin y.length ≃ Fin x.length where
  toFun j := ⟨toPerm.app f j, toPerm_app_lt_of_lt _ _ j.prop⟩
  invFun j := ⟨(toPerm.app f).symm j, by simpa using toPerm_app_lt_of_lt (inv f) _ j.prop⟩
  left_inv j := by simp
  right_inv j := by simp
\end{leancode}
This function respects composition and inverses.

Motivated by this construction, given a symmetric list \( L \), we will call the \emph{set of indices} of \( L \)
the set \(\{0, \ldots, \mathrm{length}(L) - 1\}\). On the Lean side, we will refer to \lean{Fin L.length} as the
\emph{type of indices} of \lean{L}. Thanks to Theorem~\ref{w_faithful}, we obtain
\begin{leancode}
theorem SList.getElem_toList_toEquiv {x y : SList C} (f : x ⟶ y) (i : Fin y.length) :
     y.toList[i] = x.toList[(toEquiv f) i] := …

theorem SList.hom_eq_iff_toEquiv_eq {x y : SList C} (f g : x ⟶ y) :
    f = g ↔ (toEquiv f) = (toEquiv g) where …
\end{leancode}
and we can further cement the link between permutations (or rather, bijections of sets of indices)
and symmetric lists by performing an induction to show
\begin{leancode}
theorem SList.exists_lift_equiv {x y : SList C} (φ : Fin y.length ≃ Fin x.length)
    (hφ : ∀ i : Fin y.length, y.toList[i] = x.toList[φ i]) :
    ∃ f : x ⟶ y, toEquiv f = φ := …

def SList.liftEquiv {x y : SList C} (φ : Fin y.length ≃ Fin x.length)
    (hφ : ∀ i : Fin y.length, y.toList[i] = x.toList[φ i]) :
    x ⟶ y :=
  (exists_lift_equiv φ hφ).choose

lemma SList.toEquiv_liftEquiv {x y : SList C} (φ : Fin y.length ≃ Fin x.length)
    {hφ : ∀ i : Fin y.length, y.toList[i] = x.toList[φ i]} :
    toEquiv (liftEquiv φ hφ) = φ :=
  exists_lift_equiv φ hφ |>.choose_spec
\end{leancode}

This last batch of declarations along with \lean{SList.hom_eq_iff_toEquiv_eq} from the listing before provides in
practice all of the API we need for working with symmetric lists. Altogether, this formalizes the following result:
\begin{corollary}\label{equiv_type_indx}
  Let \(L_1\) and \(L_2\) be objects of \( \mathrm{SList}(C) \). The function that sends a morphism \(L_1 \to L_2\) to
  the corresponding bijection between the associated sets of indices realizes a bijection between morphisms \(L_1 \to L_2\)
  and bijections \[\phi : \{0, \ldots, \mathrm{length}(L_2) - 1\} \iso \{0, \ldots, \mathrm{length}(L_1) - 1\}\]
  such that \(L_1[\phi(i)] = L_2[i]\).
\end{corollary}

Together, the definitions in this subsection construct a functor
\[ \mathrm{I}_C : \mathrm{SList}\left(C\right) \to \mathrm{Core}\left(\mathrm{Fin}_{/C}\right) \] by sending a symmetric
list \(L\) to its set of indices, along with the map \(i \mapsto L[i]\), and sending a map \(f\) to
the inverse of the equivalence between index types induced by \lean{toEquiv}.
The results above assert that this functor is fully faithful.
Essential surjectivity can be proven easily using that a finite set \(J\) is
equivalent to \(\mathrm{Fin}\left(k\right)\) where $k$ is the cardinality of $J$. Hence, the functor \(\mathrm{I}_C\) is
an equivalence of categories, though we cannot say anything about the symmetric monoidal aspect of this equivalence yet.

\subsection{From symmetric lists to the coherence theorem}\label{slist-vs-fsmc}

The remaining step to prove a statement somewhat resembling Theorem~\ref{coherence_theorem} is to link symmetric lists to
the free monoidal category on a set. This is part of the content of S. Piceghello's thesis
\cite{Piceghello}, which he formalized in the Rocq proof assistant within the HoTT library.
\begin{theorem}{\cite[Cor. 4.49]{Piceghello}}\label{slist_fsmc}
  Let \(C\) be a set, the category \(\mathrm{SList}\left(C\right)\) of symmetric lists on \(C\) is equivalent to the
  free symmetric monoidal category on \(C\).
\end{theorem}

Through this equivalence, the category \(\mathrm{SList}\left(C\right)\) admits an essentially unique symmetric monoidal
structure in such a way that the equivalence of Theorem~\ref{slist_fsmc} becomes a symmetric monoidal equivalence of
categories. In fact, defining this monoidal structure is part of proving the equivalence. We essentially translate the
relevant constructions from \cite{Piceghello} in our context.

\paragraph{Free symmetric monoidal categories}

The free monoidal category on a set was already formalized in \mathlib as a category presented by generators and relations.
In this presentation, objects are inductively generated by a unit object, objects coming from the base type, and a binary tensor product
operation. Morphisms are then inductively generated by formal associators, unitors, and tensor products of morphisms.
Finally, relations are added to ensure that associators and unitors are natural, that tensor products are functorial, and that
the triangle and pentagon identities hold. With such a definition, the fact that the resulting object is a
free monoidal category on the input type is essentially tautological.

In our formalization, we merely adapted this definition to the symmetric monoidal case by adding the relevant
extra generators and relations: the type of objects remains
constructed the same way, generators are added for the braiding isomorphisms, with
extra relations for the hexagon and symmetry identities.

As in the non-symmetric case, it is straightforward to show that this construction has
the expected universal property: every function from the base type to any symmetric monoidal category
\smcLinkRange{SymmMonCoherence/FreeSMC.lean}{409}{412}{lifts} to a
\smcLinkRange{SymmMonCoherence/FreeSMC.lean}{414}{430}{symmetric monoidal} functor out of the free symmetric monoidal
category we constructed. Similarly, monoidal natural transformations of symmetric monoidal functors
(or, more generally, from a symmetric monoidal functor to a lax braided functor\footnote{
  In fact, we only need the source functor to be oplax braided, i.e., an oplax monoidal functor compatible with the
  braiding, but \mathlib currently lacks type-classes for braided structure on oplax monoidal functors.}) out of this free
symmetric monoidal category can be \smcLinkRange{SymmMonCoherence/FreeSMC.lean}{472}{518}{defined} from their components at
objects coming from the base category, and are \smcLinkRange{SymmMonCoherence/FreeSMC.lean}{618}{633}{uniquely determined} by these components.

\paragraph{The symmetric monoidal structure on symmetric lists}
One of the main steps for proving Theorem~\ref{slist_fsmc} consists of defining a monoidal structure on the category of
symmetric lists.
To define the tensor product bifunctor, one lifts the operation of appending lists to a bifunctorial operation on
\(\mathrm{SList}\left(C\right)\). This is done using Lemma~\ref{up-slist} and implemented through the
\lean{RecursiveFunctorData} structure that we described in
subsection~\ref{pres-slist}, which allows one to define functors out of symmetric lists in an inductive way.

\begin{leancode}
abbrev SList.appendRecData : RecursiveFunctorData C (SList C ⥤ SList C) where
  nilObj := 𝟭 (SList C)
  consF c := Functor.whiskeringRight _ _ _|>.obj (c>~) -- c>~ is the functor c ::~ -
  swapIso x y := NatIso.ofComponents
    (fun F ↦
      (Functor.associator ..) ≪≫
        (Functor.isoWhiskerLeft F <| swapNatIso _ _) ≪≫
        (Functor.associator ..).symm)
    (fun {x y} f ↦ by ext; simp [swap_natural])
  …

def SList.appendFunctor : SList C ⥤ (SList C ⥤ SList C) := appendRecData.functor
\end{leancode}
Notice that this definition is the direct categorification of the usual inductive definition of appending two lists,
which shows that \lean{appendFunctor} is the usual operation on the underlying lists.

Extending this operation to a (not yet symmetric) \smcLink{SymmMonCoherence/SList/Monoidal.lean}{474}{monoidal structure}
on \(\mathrm{SList}\left(C\right)\) is easily done, as this operation is strictly associative and unital\footnote{This is one of the very
few places where we make use of the fact that the isomorphisms \(\upsilon_{\mathrm{nil}}\) and \(\upsilon_{\mathrm{cons}}\) that appear in Lemma~\ref{up-slist} are underlain by equalities. This is used as a shortcut to skip some computations, and is not essential in any way: we could also
follow the same computations as in~\cite{Piceghello} to avoid making use of any equality of objects.}.
The definition of the braiding is slightly more involved,
and we follow the exact same constructions as in \cite[Def. 4.25]{Piceghello}.
One first inductively defines an isomorphism
\[ Q_{x, l_1, l_2} : \left(x \dblcolon l_1\right) \otimes l_2 \simeq l_1 \otimes \left(x \dblcolon l_2\right), \]
that is natural in \(l_1\), by using the induction principle on \(l_1\) for natural transformations out of the category
of symmetric lists (Lemma~\ref{up-slist}). The base case is the isomorphism
\begin{align*}
  (x \dblcolon [\ ]) \otimes l_2 &\iso x \dblcolon ([\ ] \otimes l_2)&\text{characterization of \( \otimes \)}\\
    &\iso x \dblcolon l_2&\text{right unitor}.\\
\end{align*}
Then, given some isomorphism \(\left(x \dblcolon l_1\right) \otimes l_2 \iso l_1 \otimes \left(x \dblcolon l_2\right) \) and
\( y \in C \), one constructs \( Q_{x, y\dblcolon l_1, l_2} \) as the composition
\begin{align*}
  (x \dblcolon y \dblcolon l_1) \otimes l_2 &\iso x \dblcolon ((y \dblcolon l_1) \otimes l_2)&\text{characterization of \( \otimes \)}\\
      &\iso x \dblcolon y \dblcolon (l_1 \otimes l_2)&\text{characterization of \( \otimes \)}\\
      &\iso y \dblcolon x \dblcolon (l_1 \otimes l_2)&\text{using \(\mathrm{sw}_{x,y, l_1 \otimes l_2}\)}\\
      &\iso y \dblcolon ((x \dblcolon l_1) \otimes l_2)&\text{characterization of \( \otimes \)}\\
      &\iso y \dblcolon (l_1 \otimes (x \dblcolon l_2))&\text{induction}\\
      &\iso (y \dblcolon l_1) \otimes (x \dblcolon l_2)&\text{characterization of \( \otimes \).}\\
\end{align*}
One then \smcLinkRange{SymmMonCoherence/SList/Monoidal.lean}{669}{694}{inductively builds}
the full braiding out of \(Q\). The base case is again easily provided,
and assuming that some isomorphism \(l_1 \otimes l_2 \iso l_2 \otimes l_1\) is
constructed and \( x \in C\), one builds an isomorphism \((x \dblcolon l_1) \otimes l_2 \iso l_2 \otimes (x \dblcolon l_1)\) as the composition
\begin{align*}
  (x \dblcolon l_1) \otimes l_2 &\iso x \dblcolon (l_1 \otimes l_2)&\text{characterization of \( \otimes \)}\\
    &\iso x \dblcolon (l_2 \otimes l_1)&\text{induction}\\
    &\iso (x \dblcolon l_2) \otimes l_1&\text{characterization of \( \otimes \)}\\
    &\iso l_2 \otimes (x \dblcolon l_1)&\text{via \(Q_{x, l_2, l_1}\)}.\\
\end{align*}

\paragraph{The inclusion functor and the equivalence}
By the universal property of the free symmetric monoidal category on a type, the symmetric monoidal structure
from the previous paragraph induces a symmetric monoidal
functor \(N\) from the free symmetric monoidal category on \(C\) to symmetric lists, characterized up to a unique
monoidal isomorphism by the fact that \(N(c)\) is isomorphic to the singleton symmetric list $[c]$.\footnote{
  In fact, our implementation of free symmetric monoidal categories can
ensure that the isomorphisms $N(c) \iso [c]$ are (definitional) equalities. But for the purpose of stating things,
we prefer to avoid as much as possible a language that is not invariant under equivalences of categories.
Phrasing things in terms of isomorphisms does bring an additional subtlety in the fact that the family of isomorphisms
\(\left(N(c) \iso [c]\right)_{c \in C}\) should be considered as additional data characterizing
\(N\) in the category of symmetric monoidal functors and monoidal natural transformations.
Such families could \emph{a priori} have non-trivial automorphisms yielding non-trivial monoidal automorphisms
of \(N\).
In this particular case, there are no such automorphisms, as the objects \([c] \in \mathrm{SList}(C)\) do not have any
non-trivial automorphism and hence \(N\) is characterized up to a unique monoidal isomorphism by the mere existence of
such isomorphisms.}

In the other direction, one \smcLinkRange{SymmMonCoherence/SList/Equivalence.lean}{62}{86}{defines} a functor \(J\) from
symmetric lists to the free symmetric monoidal category on \(C\) using the induction principle for functors out of symmetric lists,
i.e., Lemma~\ref{up-slist} and the structure \lean{RecursiveFunctorData} that we introduced earlier.
The base case sends the empty list to the unit for the symmetric monoidal structure, and sends the functor
\(L \mapsto c\dblcolon L\) to the functor \( x \mapsto c \otimes x\).
Through direct computations (and several uses of \lean{SList.recNatTrans}, or rather its natural isomorphism
variant \lean{SList.recNatIso}), we can put a monoidal structure on \(J\) as well.
\smcLinkRange{SymmMonCoherence/SList/Equivalence.lean}{448}{519}{Compatibility with the braiding} is arguably the most computational part,
as it ends up relying on the recursive definition of the braiding in symmetric lists in terms of the partial
braiding \(Q\) mentioned above.

The \smcLinkRange{SymmMonCoherence/SList/Equivalence.lean}{528}{609}{unit isomorphism} is again constructed by induction,
leveraging the characterizing properties of the
functors \( J \) and \(N\): the case of an empty list is easy, and assuming an isomorphism \(N(J(L)) \iso L\) is
already constructed, its component at a symmetric list of the form \(c \dblcolon L\) is characterized as the
following composite isomorphism \begin{align*}
  N\left(J\left(c \dblcolon L\right)\right)
    &\iso N\left(c \otimes J\left(L\right)\right) &\text{by characterization of \(J\)}\\
    &\iso N\left(c\right) \otimes N\left(J\left(L\right)\right) &\text{by monoidality of \(N\)}\\
    &\iso [c] \otimes N\left(J\left(L\right)\right) &\text{by characterization of \(N\)}\\
    &\iso [c] \otimes L &\text{by induction}\\
    &\iso c\dblcolon L &\text{by definition}.\\
\end{align*}
The fact that this construction respects swaps boils down to a direct computation. One can furthermore show
that the unit isomorphism constructed this way is monoidal.

The counit isomorphism for the equivalence is easily defined, as it is supposed to be a monoidal natural
isomorphism between symmetric monoidal functors out of the free symmetric monoidal category on \(C\):
such a natural isomorphism is determined by its components at elements of \(C\), i.e., it
suffices to give a family of isomorphisms \(\left(J(N(c)) \iso c\right)_{c \in C}\), and such
isomorphisms are provided by the definitions of \( J \) and \( N \).
One can then use a small trick to show
that the equivalence constructed this way forms an adjoint equivalence:
\begin{lemma}
  Let \(F, G : \mathsf{C} \to \mathsf{D}\) be lax monoidal functors between monoidal categories, given
  monoidal isomorphisms \(\eta : F \circ G \simeq \mathrm{Id}\) and \(\epsilon : G \circ F \simeq \mathrm{Id}\),
  so that \((F, G, \eta, \epsilon)\) forms a (not necessarily adjoint) equivalence of categories.
  The counit of the adjointification \((F, G, \eta, \epsilon')\) of \((F, G, \eta, \epsilon)\) is monoidal.
\end{lemma}
\begin{proof}
The adjointification process can be performed in any bicategory, in particular, it can be performed in the bicategory of
monoidal categories, lax monoidal functors and monoidal natural transformations.
\end{proof}
In Lean, the proof of the above lemma is essentially the same:
being a monoidal natural transformation is stated as a type class \lean{NatTrans.IsMonoidal} and instances
of this typeclass are provided for left and right whiskerings of monoidal natural transformations by a lax monoidal functor,
as well as for the associators and unitors. The adjointification process only involves such constructions, so the Lean proof that
adjointification preserves monoidality is completely automated by the typeclass inference system.

By adjointifying the equivalence formed by \(J\) and \(N\), we get an adjoint equivalence in which the counit is a monoidal
natural isomorphism, and it suffices to show that this new counit is the same as the one we constructed above.
Since the new counit of the adjoint equivalence is monoidal, it is uniquely determined by its components
\(J(N(c)) \iso c\) at objects \(c \in C\). Using merely the fact that \(N\) is fully faithful
(which does not require any adjointness) and is such that \(N(c) \simeq [c]\), we can show that such a
family of isomorphisms must be unique, as symmetric lists of the form \([c]\) have no automorphisms.

\subsubsection{Symmetric lists and finite types}
To guide automation and streamline proofs related to the monoidal structure on symmetric lists,
it is useful to define proxy equivalences
\begin{leancode}
def SList.Ψ (x y : SList C) : Fin x.length ⊕ Fin y.length ≃ Fin (x ⊗ y).length := …

def SList.Φ (x : C) (l : SList C) : Unit ⊕ Fin l.length ≃ Fin (x ::~ l).length := …
\end{leancode}
Under the hood, the equivalence \lean{SList.Ψ} is a cast along the equality
\[\mathrm{length}(x \otimes y) = \mathrm{length}(x) + \mathrm{length}(y)\]
followed by the equivalence (from \mathlib) \lean{finSumFinEquiv : Fin m ⊕ Fin n ≃ Fin (m + n)}, and
similarly for \lean{SList.Φ}.
Abstracting these equivalences via definitions rather than using their underlying nature as casts
gives a very systematic way of proving equalities of morphisms out of tensor products of symmetric lists. To prove an
equality \lean{f = g} where \lean{f, g : x ⊗ y ⟶ _} are morphisms of symmetric lists, one first uses
\lean{hom_eq_iff_toEquiv_eq}, then one calls the \lean{ext} tactic to introduce a term
\lean{i : Fin (x ⊗ y).length} in context,
and then one uses surjectivity of \lean{SList.Ψ} to obtain a term \lean{Fin x.length ⊕ Fin y.length} on which the induction
principle for sum types can be used, reducing to proving the two equalities
\begin{leancode}
∀ i : Fin x.length, (toEquiv f) (Ψ x y (.inl i)) = (toEquiv g) (Ψ x y (.inl i))
\end{leancode}
and
\begin{leancode}
∀ i : Fin y.length, (toEquiv f) (Ψ x y (.inr i)) = (toEquiv g) (Ψ x y (.inr i))
\end{leancode}
This design becomes very efficient at automating proofs involving
the monoidal structure on symmetric lists when paired with a systematic characterization of the
various basic operations of the monoidal structure, such as tensor products of morphisms, in terms of
\lean{toEquiv} and \lean{Ψ}.
For instance, we can show key formulas like
\begin{leancode}
lemma whiskerRight_apply_right {x y : SList C} (f : x ⟶ y) (z : SList C)
    (i : Fin y.length) :
    toEquiv (f ▷ z) (Ψ y z <| .inl i) = Ψ _ _ (.inl (toEquiv f i)) := by …
\end{leancode}
and marking lemmas of similar shape with attributes like \lean{@[simp, grind =]} enables automation tactics to use them.
All formulas of this kind are proved by induction on symmetric lists, using the inductive definition of
the tensor product of symmetric lists and the defining property of \lean{toEquiv} as induced by the
permutation of \(\mathbb{N}\) corresponding to the label of a morphism, which acts exactly as
expected on generating morphisms.

The formula in the code listing above is one of many that, show that through the equivalence \lean{SList.Ψ},
the equivalences between types of indices induced by a tensor product of morphisms act on each factor in the expected way,
and in fact proves that through the correspondence between morphisms of symmetric lists and
equivalences between their types of indices, the braiding we constructed via Piceghello's method indeed
corresponds to swapping the factors of a sum type (i.e., the braiding in the cocartesian monoidal structure on types).

The last remaining part of Theorem~\ref{coherence_theorem} hence comes from the fact that the
functor from symmetric lists on \(C\) to the groupoid of finite types over \(C\) is symmetric monoidal.
As we can bundle the equivalence \lean{SList.Ψ} to an equivalence lying over \(C\),
the equivalence \lean{SList.Ψ} itself \emph{is} the monoidal structure isomorphism on the
functor that sends a symmetric list to its type of indices.
Naturality of this equivalence and its compatibility with associators, unitors, and braidings is exactly the content of the
characterizing lemmas on these objects that we mentioned above, and thus
the \smcLink{SymmMonCoherence/SList/EquivFintypeGrpd.lean}{585}{Lean formalization} of the fact that the equivalence
\[\mathrm{I}_C : \mathrm{SList}(C) \to \mathrm{Core}(\mathrm{Fin}_{/C}) \]
is a symmetric monoidal equivalence of categories
can be mainly discharged to automation tactics. This finishes the proof of Theorem~\ref{coherence_theorem}.

\section{Pseudofunctors out of the pith of bicategories of spans}\label{section-pseudo-out-of-pith}
We now turn to the second aspect of our work, which is to use the coherence theorem to produce unbiased symmetric
monoidal categories as in Theorem~\ref{main_th}.
\subsection{Bicategories of spans}
The statement of Theorem~\ref{main_th} involves bicategories of spans.
We briefly recall their definitions and explain how we implement them, as they are currently
absent from \mathlib.

Given two classes of morphisms \(W_l\) and \(W_r\) in a category \(C\), one can consider spans with left legs in
\(W_l\) and right legs in \(W_r\): these are by definition diagrams in \(C\)
\[\begin{tikzcd}
  & {\widehat{S}} \\
  {c} && {c'}
  \arrow["{l}"', from=1-2, to=2-1]
  \arrow["{r}", from=1-2, to=2-3]
\end{tikzcd}\]
such that \(l \in W_l\) and \(r \in W_r\). This definition can be formalized as follows:
\begin{leancode}
structure Span {C : Type*} [Category* C]
    (Wₗ : MorphismProperty C) (Wᵣ : MorphismProperty C)
    (c c' : C) where
  apex : C
  l : apex ⟶ c
  r : apex ⟶ c'
  wl : Wₗ l
  wr : Wᵣ r
\end{leancode}
Spans admit a category structure, where a morphism \(S \to S'\) is a morphism between the apices
of the spans that is compatible with the projections.

We will say the classes of morphisms \(W_l\) and \(W_r\) are \emph{adequate} if they satisfy the following conditions:
\begin{itemize}
  \item They both contain identities and are stable under compositions.
  \item \(W_l\) has pullbacks against \(W_r\), i.e., for every diagram
    \[\begin{tikzcd}[cramped]
              x && z \\
              & y
              \arrow["f"', from=1-1, to=2-2]
              \arrow["g", from=1-3, to=2-2]
      \end{tikzcd}\]
  such that \(f \in W_r\) and \(g \in W_l\), there exists a pullback of \( g \) along \( f \) in \(C\).
\item \(W_l\) and \(W_r\) are stable under base change against each other, i.e., if
  \[\begin{tikzcd}[cramped]
          & t \\
          x && z \\
          & y
          \arrow["{f'}"', from=1-2, to=2-1]
          \arrow["{g'}", from=1-2, to=2-3]
          \arrow["\lrcorner"{anchor=center, pos=0.125, rotate=-45}, draw=none, from=1-2, to=3-2]
          \arrow["f"', from=2-1, to=3-2]
          \arrow["g", from=2-3, to=3-2]
  \end{tikzcd}\]
  is a pullback square in \(C\) where \(f \in W_r\) and \(g \in W_l\), then \(f' \in W_l\) and \(g' \in W_r\).
\end{itemize}
When the classes of morphisms \(W_l\) and \(W_r\) are adequate, one can define a bicategory of spans:
\begin{definition}\label{def-spans}
  Let \(C\) be a category and let \(W_l, W_r\) be adequate classes of morphisms. The \emph{bicategory of spans} of \(C \) with respect
  to \(W_l\) and \(W_r\), denoted by \(\mathrm{Span}(C, W_l, W_r)\), is the bicategory defined by the following data:
  \begin{itemize}
  \item Objects are objects of \(C\).
  \item Morphisms from \(c\) to \(c'\) are spans \(c \xleftarrow{f} \widehat{S} \xrightarrow{g} c'\).
  \item 2-morphisms are morphisms of spans.
  \item The identity morphism on \(c\) is the span \(c \xleftarrow{\mathrm{Id}_c} c \xrightarrow{\mathrm{Id}_c} c\).
  \item Composition of two spans \(c \xleftarrow{f} \widehat{S} \xrightarrow{g} c'\) and \(c' \xleftarrow{h} \widehat{T} \xrightarrow{k} c''\)
  is defined as the total span in the following diagram
  
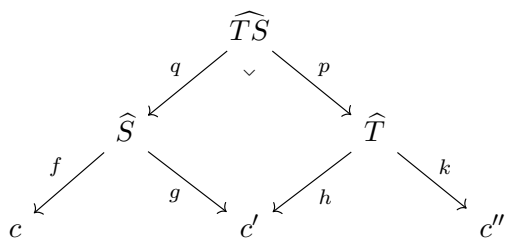
\[\begin{tikzcd}
          && {\widehat{TS}} \\
          & {\widehat{S}} && {\widehat{T}} \\
          c && {c'} && {c''}
          \arrow["q"', from=1-3, to=2-2]
          \arrow["p", from=1-3, to=2-4]
          \arrow["\lrcorner"{anchor=center, pos=0.125, rotate=-45}, draw=none, from=1-3, to=3-3]
          \arrow["f"', from=2-2, to=3-1]
          \arrow["g"', from=2-2, to=3-3]
          \arrow["h", from=2-4, to=3-3]
          \arrow["k", from=2-4, to=3-5]
  \end{tikzcd}\]
  where the square is a pullback square.
  \item Horizontal composition is induced by the universal property of pullbacks.
  \item Associator and unitor isomorphisms are defined as the canonical morphisms between pullbacks.
  \end{itemize}
\end{definition}
When \(W_l\) and \(W_r\) consist of all morphisms, adequacy means that \(C\) has pullbacks, and we denote the corresponding bicategory
of spans by \(\mathrm{Span}(C)\) in this case.

The most notable point when \smcLinkRange{SymmMonCoherence/Spans/Basic.lean}{398}{417}{implementing} this bicategory structure is the
quirk that the implementation uses an arbitrary pullback provided by the axiom of choice for the definition of the composition
of spans.
To make the actual pullback square underlying a composition in \(\mathrm{Span}\) as irrelevant as possible
and to reduce the amount of possible leakage from \mathlib's \lean{Limits.HasLimit} API (which provides the pullbacks via the
axiom of choice),
we introduce definitions
\begin{leancode}
variable {X Y Z : Spans C Wₗ Wᵣ} (S₁ : X ⟶ Y) (S₂ : Y ⟶ Z)
def Spans.πₗ : (S₁ ≫ S₂).apex ⟶ S₁.apex := Limits.pullback.fst _ _
def Spans.πᵣ : (S₁ ≫ S₂).apex ⟶ S₂.apex := Limits.pullback.snd _ _

def Spans.compPullbackCone : Limits.PullbackCone S₁.r S₂.l :=
  Limits.PullbackCone.mk (πₗ _ _) (πᵣ _ _) (comp_comm _ _)

def Spans.isLimitCompPullbackCone : Limits.IsLimit (compPullbackCone S₁ S₂) := …

def Spans.compLiftApex {c : C} (fₗ : c ⟶ S₁.apex) (fᵣ : c ⟶ S₂.apex) (hₘ : fₗ ≫ S₁.r = fᵣ ≫ S₂.l) :
    c ⟶ (S₁ ≫ S₂).apex :=
  Limits.PullbackCone.IsLimit.lift …
\end{leancode}
and make these objects the primary API for working with composition of spans.
For instance, instead of using \mathlib's existing \lean{CategoryTheory.Limits.pullbackAssoc},
the associators and unitors for the bicategory structure are redefined using \lean{Spans.compLiftApex} and
are characterized via \lean{Spans.πₗ} and \lean{Spans.πᵣ}, e.g.,
\begin{leancode}
lemma associator_hom_hom_πₗ {W X Y Z : Spans C Wₗ Wᵣ} (S₁ : W ⟶ X) (S₂ : X ⟶ Y) (S₃ : Y ⟶ Z) :
    (α_ S₁ S₂ S₃).hom.hom ≫ πₗ .. = πₗ .. ≫ πₗ .. := …
\end{leancode}
A design like this minimizes the amount of refactoring one would need to make to use a bundled family of chosen pullbacks instead of
the ones provided by the axiom of choice.
A systematic characterization of the associators and unitors via \lean{Spans.πₗ} and \lean{Spans.πᵣ} also creates an
environment that helps guide automation tactics like \lean{simp} and \lean{grind}.
\subsection{Building pseudofunctors out of the pith of spans}
In the following, we will now assume that \(C\) is a category with pullbacks, and we will assume that the left and right
classes for our spans consist of all morphisms (so that adequacy reduces to the fact that \(C\) has pullbacks).
Given a span \(S : c \xleftarrow{f} \widehat{S} \xrightarrow{g} c'\) in \(C\), the diagram
\[\begin{tikzcd}[cramped]
	&& {\widehat{S}} \\
	& {\widehat{S}} && {\widehat{S}} \\
	c && {\widehat{S}} && {c'}
	\arrow[equals, from=1-3, to=2-2]
	\arrow[equals, from=1-3, to=2-4]
	\arrow["\lrcorner"{anchor=center, pos=0.125, rotate=-45}, draw=none, from=1-3, to=3-3]
	\arrow["f"', from=2-2, to=3-1]
	\arrow[equals, from=2-2, to=3-3]
	\arrow[equals, from=2-4, to=3-3]
	\arrow["g", from=2-4, to=3-5]
\end{tikzcd}\]
exhibits \(S\) as isomorphic to the composition of two smaller spans
\(f^* : c \xleftarrow{f} \widehat{S} \xrightarrow{\mathrm{Id}_{\widehat{S}}} \widehat{S} \) and
\(g_! : \widehat{S} \xleftarrow{\mathrm{Id}_{\widehat{S}}} \widehat{S} \xrightarrow{g} c' \). The constructions
\(f \mapsto f^*\) and \(f \mapsto f_!\) are \smcLinkRange{SymmMonCoherence/Spans/Inclusions.lean}{24}{62}{both pseudofunctorial}, the first one being contravariant and the second one being covariant.

Internally to the bicategory of spans, the 1-morphism \(f_!\) is left adjoint to the
morphism \(f^*\): this can be seen by contemplating the diagrams
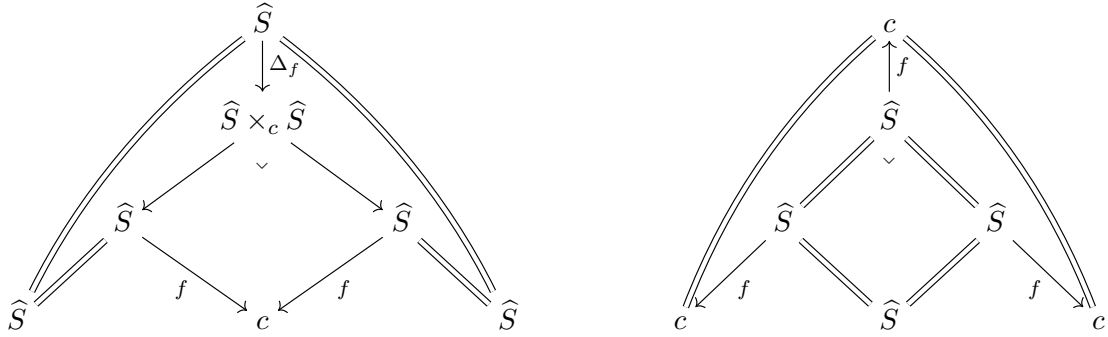
\[\begin{tikzcd}[cramped]
	&& {\widehat{S}} &&&&&& c \\
	&& {\widehat{S} \times_c \widehat{S}} &&&&&& {\widehat{S}} \\
	& {\widehat{S}} && {\widehat{S}} &&&& {\widehat{S}} && {\widehat{S}} \\
	{\widehat{S}} && c && {\widehat{S}} && c && {\widehat{S}} && c
	\arrow["{\Delta_f}", from=1-3, to=2-3]
	\arrow[curve={height=12pt}, equals, from=1-3, to=4-1]
	\arrow[curve={height=-12pt}, equals, from=1-3, to=4-5]
	\arrow[curve={height=12pt}, equals, from=1-9, to=4-7]
	\arrow[curve={height=-12pt}, equals, from=1-9, to=4-11]
	\arrow[from=2-3, to=3-2]
	\arrow[from=2-3, to=3-4]
	\arrow["\lrcorner"{anchor=center, pos=0.125, rotate=-45}, draw=none, from=2-3, to=4-3]
	\arrow["f"', from=2-9, to=1-9]
	\arrow[equals, from=2-9, to=3-8]
	\arrow[equals, from=2-9, to=3-10]
	\arrow["\lrcorner"{anchor=center, pos=0.125, rotate=-45}, draw=none, from=2-9, to=4-9]
	\arrow[equals, from=3-2, to=4-1]
	\arrow["f"', from=3-2, to=4-3]
	\arrow["f", from=3-4, to=4-3]
	\arrow[equals, from=3-4, to=4-5]
	\arrow["f", from=3-8, to=4-7]
	\arrow[equals, from=3-8, to=4-9]
	\arrow[equals, from=3-10, to=4-9]
	\arrow["f"', from=3-10, to=4-11]
\end{tikzcd}\]
which we can interpret as defining the unit \(\mathrm{Id}_{\widehat{S}} \to f^*f_!\) and counit
\(f_!f^* \to \mathrm{Id}_{c}\) morphisms of an adjunction. This adjunction is also pseudofunctorial with
respect to \(f\), i.e., the construction lifts to a pseudofunctor from \(C\) to a suitable bicategory of
adjunctions in \(\mathrm{Span}(C)\).

Furthermore, given a pullback square
\[\begin{tikzcd}
  {x} & {y} \\
  {z} & {t}
  \arrow["{c_0}", from=1-1, to=1-2]
  \arrow["{c_1}"', from=1-1, to=2-1]
  \arrow["\lrcorner"{anchor=center, pos=0.125}, draw=none, from=1-1, to=2-2]
  \arrow["{c_2}", from=1-2, to=2-2]
  \arrow["{c_3}"', from=2-1, to=2-2]
\end{tikzcd}\]
in \(C\), we can look at the diagram
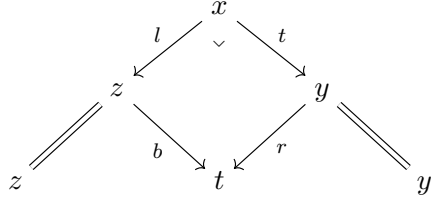
\[\begin{tikzcd}[cramped]
    && x \\
    & z && y \\
    z && t && y
    \arrow["l"', from=1-3, to=2-2]
    \arrow["t", from=1-3, to=2-4]
    \arrow["\lrcorner"{anchor=center, pos=0.125, rotate=-45}, draw=none, from=1-3, to=3-3]
    \arrow[equals, from=2-2, to=3-1]
    \arrow["b"', from=2-2, to=3-3]
    \arrow["r", from=2-4, to=3-3]
    \arrow[equals, from=2-4, to=3-5]
\end{tikzcd}\]
as an isomorphism of 1-morphisms
  \[
    \left(x \xleftarrow{l} x \xrightarrow{t} y\right) \simeq \left(r^*b_!\right)
  \]
and we have seen that the first 1-morphism is also isomorphic to \(t_!l^*\). Hence, the pullback square above
can be interpreted as an isomorphism \[r^*b_! \simeq t_!l^*\] which looks like many ``base-change formulas''.\footnote{
  In the usual presentation of base change formulas,
one direction of the isomorphism is the canonical isomorphism obtained via calculus of adjunctions.
In this setting of spans, we can verify that the morphism we constructed is indeed that
morphism for the adjunctions and adjointable squares at hand.} While these general facts are not strictly needed
for our purpose of unbiasing symmetric monoidal categories, we still
\smcLinkRange{SymmMonCoherence/Spans/Adjunction.lean}{163}{190}{implement them} as a verification of the robustness
of our implementation of span bicategories.

Since adjunctions internal to bicategories transport through pseudofunctors, all of these constraints on adjunctions,
as well as the base change formula, must hold in the target bicategory. This gives a rather strong constraint
on the restriction of the pseudofunctors along the inclusions from base categories, and it is a classical theorem
(see e.g., \cite[Th. A.2]{HERMIDA2000164}) that these constraints in fact characterize pseudofunctors out of spans bicategories,
i.e., given a bicategory \(B\),
a pseudofunctor \(\mathrm{Span}\left(C\right) \to B\) can be defined out of the data of two pseudofunctors
\begin{align*}U : C \to B \\
  V : C^{\mathrm{op}} \to B
\end{align*}
such that \(U(x) = V(x)\) for all \(x\), such that \(U\) is left adjoint to \(V\) in a suitably pseudofunctorial
sense, and such that the base change morphisms induced by pullback squares via this data are isomorphisms.

In the situation of Theorem~\ref{main_th}, we are interested in defining pseudofunctors out of the pith of
the bicategory of spans, i.e., the sub-bicategory of spans one obtains by only keeping morphisms of spans where
the morphism between the apices is an isomorphism. We cannot expect the pseudofunctor that appears in
Theorem~\ref{main_th} to be the restriction of a pseudofunctor defined on the whole bicategory of spans, as in general
the functors \(f_!\) and \(f^*\) that appear in the theorem are not adjoint to each other\footnote{
  For instance, when \(f\) is the unique morphism \(\{0, 1\} \to \{*\}\), the functor \(f_!\) in Theorem
\ref{main_th} identifies with the tensor product \emph{bifunctor},
which in general does not preserve colimits as a bifunctor (it sometimes preserves colimits separately in
each variable, which is different) and hence cannot be a left adjoint.}.
We can see that the adjunction \( f_! \dashv f^* \) we described earlier does \emph{not} lift to the pith
unless the morphism \(f\) is an isomorphism. Yet, a pullback square still defines isomorphisms
\(r^*b_! \simeq t_!l^*\); but these isomorphisms no longer interpret
as a general base change morphism obtained via calculus of adjunctions internal to the pith of the bicategory.
Hence, to describe a pseudofunctor out of \(\mathrm{Pith}\left(\mathrm{Span}\left(C\right)\right)\)
through its restrictions along the inclusions of \(C\), one must add the base-change isomorphisms for pullback squares
as external extra data. Hence, we introduce the following definition:
\begin{definition}\label{def-pbc}
  Let \(C\) be a category and \(B\) be a bicategory.
  A \(B\)-valued \emph{Pith-Beck-Chevalley} system on \(C\) consists of the following data:
  \begin{itemize}
    \item A pseudofunctor \(U : C \to B\) and \(V : C^{\mathrm{op}} \to B\), such that \(U\) and \(V\) agree on objects.
    \item For every cartesian square
      \[\begin{tikzcd}
        {c_0} & {c_1} \\
        {c_2} & {c_3}
        \arrow["{t}", from=1-1, to=1-2]
        \arrow["{l}"', from=1-1, to=2-1]
        \arrow["\lrcorner"{anchor=center, pos=0.125}, draw=none, from=1-1, to=2-2]
        \arrow["S"{anchor=center}, draw=none, from=1-1, to=2-2]
        \arrow["{r}", from=1-2, to=2-2]
        \arrow["{b}"', from=2-1, to=2-2]
      \end{tikzcd}\]
      in \(C\), an isomorphism \[\mathrm{BC(S)} : V(b)U(r) \iso U(l)V(t) \] in \(B\).
  \end{itemize}
  The data must satisfy the following conditions:
  \begin{itemize}
    \item The isomorphisms \(\mathrm{BC}\) respect horizontal pasting of squares:
      given a diagram
      \[\begin{tikzcd}[cramped]
              a & b & c \\
              d & e & f
              \arrow["{t_0}", from=1-1, to=1-2]
              \arrow["{v_0}"', from=1-1, to=2-1]
              \arrow["L"{description}, "\lrcorner"{anchor=center, pos=0.125}, draw=none, from=1-1, to=2-2]
              \arrow["{t_1}", from=1-2, to=1-3]
              \arrow["{v_1}", from=1-2, to=2-2]
              \arrow["R"{description}, "\lrcorner"{anchor=center, pos=0.125}, draw=none, from=1-2, to=2-3]
              \arrow["{v_2}", from=1-3, to=2-3]
              \arrow["{b_0}"', from=2-1, to=2-2]
              \arrow["{b_1}"', from=2-2, to=2-3]
      \end{tikzcd},\]
      the diagram
      \[\begin{tikzcd}[cramped]
              && {U(v_0)V(t_1t_0)} \\
              {V(b_1b_0)U(v_2)} &&&& {U(v_0)\left(V(t_0)V(t_1)\right)} \\
              {\left(V(b_0)V(b_1)\right)U(v_2)} &&&& {\left(U(v_0)V(t_0)\right)V(t_1)} \\
              {V(b_0)\left(V(b_1)U(v_2)\right)} && {V(b_0)\left(U(v_1)V(t_1)\right)} && {\left(V(b_0)U(v_1)\right)V(t_1)}
              \arrow["{U(v_0)\cdot V_{\mathrm{comp}}}", from=1-3, to=2-5]
              \arrow["{\mathrm{BC}(\mathrm{hcomp}(L, R))}", from=2-1, to=1-3]
              \arrow["{V_{\mathrm{comp}} \cdot U(v_2)}"', from=2-1, to=3-1]
              \arrow["\alpha"', from=3-1, to=4-1]
              \arrow["\alpha"', from=3-5, to=2-5]
              \arrow["{V(b_0)\cdot\mathrm{BC}(R)}"', from=4-1, to=4-3]
              \arrow["{\alpha^{-1}}"', from=4-3, to=4-5]
              \arrow["{\mathrm{BC}(L)\cdot V(t_1)}"', from=4-5, to=3-5]
      \end{tikzcd}\]
      commutes.
    \item The isomorphisms \(\mathrm{BC}\) respect vertical pasting of squares:
      given a diagram
      \[\begin{tikzcd}[cramped]
              a & b \\
              c & d \\
              e & f
              \arrow["{h_0}", from=1-1, to=1-2]
              \arrow["{l_0}"', from=1-1, to=2-1]
              \arrow["T"{description}, "\lrcorner"{anchor=center, pos=0.125}, draw=none, from=1-1, to=2-2]
              \arrow["{r_0}", from=1-2, to=2-2]
              \arrow["{h_1}"', from=2-1, to=2-2]
              \arrow["{l_1}"', from=2-1, to=3-1]
              \arrow["B"{description}, "\lrcorner"{anchor=center, pos=0.125}, draw=none, from=2-1, to=3-2]
              \arrow["{r_1}", from=2-2, to=3-2]
              \arrow["{h_2}"', from=3-1, to=3-2]
      \end{tikzcd},\]
      the diagram
      \[\begin{tikzcd}
              && {U(l_1l_0)V(h_0)} \\
              {V(h_2)U(r_1r_0)} &&& {\left(U(l_1)U(l_0)\right)V(h_0)} \\
              {V(h_2)\left(U(r_1)U(r_0)\right)} &&& {U(l_1)\left(U(l_0)V(h_0)\right)} \\
              {\left(V(h_2)U(r_1)\right)U(r_0)} && {\left(U(l_1)V(h_1)\right)U(r_0)} & {U(l_1)\left(V(h_1)U(r_0)\right)}
              \arrow["{U_{\mathrm{comp}} \cdot V(h_0)}", from=1-3, to=2-4]
              \arrow["{\mathrm{BC}(\mathrm{vcomp}(T, B))}", from=2-1, to=1-3]
              \arrow["{V(h_2)\cdot U_{\mathrm{comp}}}", from=2-1, to=3-1]
              \arrow["\alpha^{-1}", from=3-1, to=4-1]
              \arrow["\alpha^{-1}"', from=3-4, to=2-4]
              \arrow["{\mathrm{BC}(B)\cdot U(r_0)}", from=4-1, to=4-3]
              \arrow["{\alpha}", from=4-3, to=4-4]
              \arrow["{U(l_1) \cdot \mathrm{BC}(T)}"', from=4-4, to=3-4]
      \end{tikzcd}\]
      commutes.
    \item The isomorphisms \(\mathrm{BC}\) respect horizontal unit squares:
      given a morphism \(f : x \to y\) in \(C\), the diagram
      \[\begin{tikzcd}[cramped]
              {V(\mathrm{Id}_y)U(f)} && {U(f)V(\mathrm{Id}_x)} \\
              {\mathrm{Id}_{\mathrm{V}(y)}U(f)} && {U(f)\mathrm{Id}_{V(x)}} \\
              & {U(f)}
              \arrow["{\mathrm{BC}}", from=1-1, to=1-3]
              \arrow["{V_{\mathrm{Id}} \cdot U(f)}"', from=1-1, to=2-1]
              \arrow["{U(f) \cdot V_{\mathrm{Id}}}", from=1-3, to=2-3]
              \arrow["\lambda"', from=2-1, to=3-2]
              \arrow["\rho", from=2-3, to=3-2]
      \end{tikzcd}\]
      commutes.
    \item The isomorphisms \(\mathrm{BC}\) respect vertical unit squares:
      given a morphism \(f : x \to y\) in \(C\), the diagram
      \[\begin{tikzcd}[cramped]
          {V(f)U(\mathrm{Id}_y)} && {U(\mathrm{Id}_x)V(f)} \\
          {V(f)\mathrm{Id}_{U(y)}} && {\mathrm{Id}_{U(x)}V(f)} \\
          & {V(f)}
          \arrow["{\mathrm{BC}}", from=1-1, to=1-3]
          \arrow["{ V(f) \cdot U_{\mathrm{Id}}}"', from=1-1, to=2-1]
          \arrow["{U_{\mathrm{Id}}\cdot V(f)}", from=1-3, to=2-3]
          \arrow["\rho"', from=2-1, to=3-2]
          \arrow["\lambda", from=2-3, to=3-2]
      \end{tikzcd}\]
      commutes.
  \end{itemize}
\end{definition}
We \smcLinkRange{SymmMonCoherence/Spans/PseudoFromBurnside/Basic.lean}{47}{164}{formalize} the above definition as follows.
\begin{leancode}
structure PseudofunctorCore
    (C : Type u₁) [Category.{v₁} C] (B : Type u₂) [Bicategory.{w₁, v₂} B]
where
  obj : C → B
  u {x y : C} : (x ⟶ y) → (obj x ⟶ obj y)
  v {x y : C} : (x ⟶ y) → (obj y ⟶ obj x)
  uId' {x : C} (f : x ⟶ x) (hf : f = 𝟙 x := by cat_disch) : u f ≅ 𝟙 (obj x)
  vId' {x : C} (f : x ⟶ x) (hf : f = 𝟙 x := by cat_disch) : v f ≅ 𝟙 (obj x)
  uComp' {x y z : C} (f : x ⟶ y) (g : y ⟶ z) (h : x ⟶ z) (hfg : f ≫ g = h := by cat_disch) :
    u h ≅ u f ≫ u g
  vComp' {x y z : C} (f : x ⟶ y) (g : y ⟶ z) (h : x ⟶ z) (hfg : f ≫ g = h := by cat_disch) :
    v h ≅ v g ≫ v f
  -- pseudofunctoriality of u
  u_associator {c₀ c₁ c₂ c₃ : C} (f : c₀ ⟶ c₁) (g : c₁ ⟶ c₂) (h : c₂ ⟶ c₃) : …
  u_left_unitor {c₀ c₁ : C} (f : c₀ ⟶ c₁) : …
  u_right_unitor {c₀ c₁ : C} (f : c₀ ⟶ c₁) : …
  -- pseudofunctoriality of v
  v_associator {c₀ c₁ c₂ c₃ : C} (f : c₁ ⟶ c₀) (g : c₂ ⟶ c₁) (h : c₃ ⟶ c₂) : …
  v_left_unitor {c₀ c₁ : C} (f : c₁ ⟶ c₀) : …
  v_right_unitor {c₀ c₁ : C} (f : c₁ ⟶ c₀) : …
  -- base change isomorphisms and their compatibilities
  baseChangeIso {c₀ c₁ c₂ c₃ : C} (t : c₀ ⟶ c₁) (l : c₀ ⟶ c₂) (r : c₁ ⟶ c₃) (b : c₂ ⟶ c₃)
    (S : IsPullback t l r b) :
    u r ≫ v b ≅ v t ≫ u l
  baseChangeIso_unit_vert {x y : C} (f : x ⟶ y) :
    (baseChangeIso f (𝟙 x) (𝟙 y) f (IsPullback.of_vert_isIso .mk)).hom =
    (uId' (𝟙 y)).hom ▷ v f ≫ (λ_ _).hom ≫ (ρ_ _).inv ≫ v f ◁ (uId' (𝟙 x)).inv
  baseChangeIso_unit_horiz {x y : C} (f : x ⟶ y) :
    (baseChangeIso (𝟙 x) f f (𝟙 y) (IsPullback.of_horiz_isIso .mk)).hom =
    u f ◁ (vId' (𝟙 y)).hom ≫ (ρ_ _).hom  ≫ (λ_ _).inv ≫ (vId' (𝟙 x)).inv ▷ u f
  baseChangeIso_comp_horiz : …
  baseChangeIso_comp_vert : …
\end{leancode}
We can observe a few technical things about this design:
\begin{enumerate}
  \item In Remark~\ref{eq_objects}, we detailed the woes of working with propositional equalities of objects in categories. This
    discussion applies equally to propositional equalities of 1-morphisms in bicategories, as they are objects of categories
    of morphisms.
    In our case, though, some equalities of 1-morphisms must appear: any 2-morphism in a locally discrete bicategory comes
    from such an equality. To avoid a proliferation of \lean{eqToHom}-type morphisms in our expressions,
    we employ a technique also used in \mathlib to deal with pseudofunctors out of locally discrete bicategories,
    which is to give definitional control over the return type of 2-isomorphisms: this is seen, for instance, in the
    field \lean{uComp'}, where the composition \lean{h} is made a parameter along with a proof that \lean{f ≫ g = h}.
  \item In order to phrase the fact that the pseudofunctors \(U\) and \(V\) agree on objects, we need to unpack
    all of their fields (\lean{uComp'}, \lean{u_associator}, etc.) in the structure to make it so that the agreement on objects is definitional. Trying
    to formulate this agreement on objects via propositional equalities of objects in the target bicategory would not be
    reasonable, as propositional equalities of objects in bicategories further amplifies the problems with dependent types
    described in Remark~\ref{eq_objects}.

    An alternative design would have been to package this part of the structure as the data of two
    pseudofunctors along with the extra data of a family of equivalences (in the bicategorical sense) between their
    objects.
    The added complexity of constantly using these extra equivalences
    to correct the source or target of the pseudofunctors every time we need to compose 1-morphisms,
    as well as the extra reassociations of compositions that become needed,
    would make this other possible approach really hard to work with in practice.
\end{enumerate}
Given \(\mathcal{S} := (U, V, \mathrm{BC})\) a Pith-Beck-Chevalley system on a category \(C\), and an isomorphism
\(\phi : x \to y\) in \(C\), we will denote by \(\eta_{\mathcal{S}}(\phi)\) the morphism
  \(U(\phi)V(\phi) \to \mathrm{Id}_y\) obtained as the composition
    \[\begin{tikzcd}[cramped]
        {U(\phi)V(\phi)} & {V(\mathrm{Id})U(\mathrm{Id})} & {\mathrm{Id}\circ\mathrm{Id}} & {\mathrm{Id}}
        \arrow["{\mathrm{BC}(S)^{-1}}", from=1-1, to=1-2]
        \arrow["{V_{\mathrm{Id}} \cdot U_{\mathrm{Id}}}", from=1-2, to=1-3]
        \arrow["\lambda", from=1-3, to=1-4]
    \end{tikzcd}\]
  where \(S\) is the cartesian square
  \[\begin{tikzcd}
    {x} & {y} \\
    {y} & {y}
    \arrow["{\phi}", from=1-1, to=1-2]
    \arrow["{\phi}"', from=1-1, to=2-1]
    \arrow["\lrcorner"{anchor=center, pos=0.125}, draw=none, from=1-1, to=2-2]
    \arrow["{\mathrm{Id}}", from=1-2, to=2-2]
    \arrow["{\mathrm{Id}}"', from=2-1, to=2-2]
  \end{tikzcd}.\]
The main result we formalize is
\begin{thdef}\label{exists-pseudo-bc-system}
  Let \(C\) be a category with pullbacks, \(B\) be a bicategory, and let \(\mathcal{S} := (U, V, \mathrm{BC})\) be a
  \(B\)-valued Pith-Beck-Chevalley
  system on \(C\). Then, there is a pseudofunctor
  \[ \mathcal{F}_{\mathcal{S}} : \mathrm{Pith}\left(\mathrm{Span}(C)\right) \to B \]
  with the following properties:
  \begin{itemize}
    \item For every \(c \in C\), \(\mathcal{F}_{\mathcal{S}}(c) = U(c)\).
    \item Given \(c, c' \in C\) and a span \(S : c \xleftarrow{f} \widehat{S} \xrightarrow{g} c'\) in \(C\),
      \[\mathcal{F}_{\mathcal{S}}(S) = U(g)V(f).\]
    \item Given \(c, c' \in C\), spans \(S : c \xleftarrow{f} \widehat{S} \xrightarrow{g} c'\) and
      \(S' : c \xleftarrow{f'} \widehat{S} \xrightarrow{g'} c'\), and given a morphism of spans \(\phi : \widehat{S} \simeq \widehat{S'}\),
      the diagram
      \[\begin{tikzcd}[cramped]
              {U(g)V(f)} && {U(g')V(f')} \\
              {\left(U(g'\phi)\right)\left(V(f'\phi)\right)} && {\left(U(g')\mathrm{Id}\right)V(f')} \\
              {\left(U(g')U(\phi)\right)\left(V(\phi)V(f')\right)} && {\left(U(g')\left(U(\phi)V(\phi)\right)\right)V(f')} \\
              & {\left(\left(U(g')U(\phi)\right)V(\phi)\right)V(f')}
              \arrow["{\mathcal{F}_{\mathcal{S}}(\eta)}", from=1-1, to=1-3]
              \arrow[equals, from=1-1, to=2-1]
              \arrow["{U_{\mathrm{comp}}\cdot V_{\mathrm{comp}}}"', from=2-1, to=3-1]
              \arrow["{\rho \cdot V(f')}"', from=2-3, to=1-3]
              \arrow["{\alpha^{-1}}", from=3-1, to=4-2]
              \arrow["{\left(U(g') \cdot \eta_{\mathcal{S}}(\phi)\right)\cdot V(f')}", from=3-3, to=2-3]
              \arrow["{\alpha \cdot V(f')}"'{pos=0.7}, from=4-2, to=3-3]
      \end{tikzcd}\]
      commutes.
    \item Given \(c, c', c'' \in C\), spans \(S : c \xleftarrow{f} \widehat{S} \xrightarrow{g} c'\) and
      \(S' : c' \xleftarrow{h} \widehat{S'} \xrightarrow{k} c''\) in \(C\). Denote by \(\pi_1\) (resp. \(\pi_2\)) the
      projection \(\widehat{S'S} \to \widehat{S}\) (resp. \(\widehat{S'S} \to \widehat{S'}\)), the diagram
      \[\begin{tikzcd}[cramped]
          {U(k\pi_2)V(f\pi_1)} &&&& {\left(U(k)V(h)\right)\left(U(g)V(f)\right)} \\
          {\left(U(k)U(\pi_2)\right)\left(V(\pi_1)V(f)\right)} \\
          {\left(\left(U(k)U(\pi_2)\right)V(\pi_1)\right)V(f)} &&&& {\left(\left(U(k)V(h)\right)U(g)\right)V(f)} \\
          {\left(U(k)\left(U(\pi_2)V(\pi_1)\right)\right)V(f)} &&&& {\left(U(k)\left(V(h)U(g)\right)\right)V(f)}
          \arrow["{\left(\mathcal{F}_{\mathcal{S}}\right)_{\mathrm{comp}, S, S'}}", from=1-1, to=1-5]
          \arrow["{U_{\mathrm{comp}}\cdot V_{\mathrm{comp}}}"', from=1-1, to=2-1]
          \arrow["{\alpha^{-1}}"', from=2-1, to=3-1]
          \arrow["{\alpha \cdot V(f)}"', from=3-1, to=4-1]
          \arrow["\alpha"', from=3-5, to=1-5]
          \arrow["{\left(U(k)\cdot\mathrm{BC}(\mathrm{Cs})^{-1}\right) \cdot V(f)}"', from=4-1, to=4-5]
          \arrow["{\alpha^{-1} \cdot V(f)}"', from=4-5, to=3-5]
      \end{tikzcd}\]
      commutes, where \(\mathrm{Cs}\) is the cartesian square
      \[\begin{tikzcd}
        {\widehat{S'S}} & {\widehat{S}} \\
        {\widehat{S'}} & {c'}
        \arrow["{\pi_1}", from=1-1, to=1-2]
        \arrow["{\pi_2}"', from=1-1, to=2-1]
        \arrow["\lrcorner"{anchor=center, pos=0.125}, draw=none, from=1-1, to=2-2]
        \arrow["{g}", from=1-2, to=2-2]
        \arrow["{g'}"', from=2-1, to=2-2]
      \end{tikzcd}.\]
    \item Given \(c \in C\), the diagram
      \[\begin{tikzcd}[cramped]
              {U(\mathrm{Id}_c)V(\mathrm{Id}_c)} && {\mathrm{Id}} \\
              & {\mathrm{Id}_c \circ \mathrm{Id}_c}
              \arrow["{(\mathcal{F}_{\mathcal{S}})_{\mathrm{Id}_c}}", from=1-1, to=1-3]
              \arrow["{U_{\mathrm{Id}_c} \cdot V_{\mathrm{Id}_c}}"', from=1-1, to=2-2]
              \arrow["\lambda"', from=2-2, to=1-3]
      \end{tikzcd}\]
      commutes.
  \end{itemize}
\end{thdef}
Note that the pseudofunctor that appears in the theorem is necessarily unique, as all of the data parts of a
pseudofunctor are fully described here. The fact that these components assemble into a pseudofunctor is purely
propositional and carries no extra data.

The \smcLink{SymmMonCoherence/Spans/PseudoFromBurnside/Pseudofunctor.lean}{531}{formalization} of
Theorem~\ref{exists-pseudo-bc-system} is obtained by bicategorical computations. The computations are rather
tedious as the diagrams involve many structural isomorphisms in bicategories (associators, unitors). The most tedious
compatibility to check is that of composition with associators in bicategories of spans,
which produces goals and expressions with a lot of bicategorical noise.
\mathlib's \mathlibLinkRange{Mathlib/Tactic/CategoryTheory/BicategoricalComp.lean}{44}{48}{\lean{bicategoricalComp}},
combined with \mathlib's
\mathlibLinkRange{Mathlib/Tactic/CategoryTheory/Bicategory/Basic.lean}{42}{53}{\lean{bicategory} tactic} is helpful,
but not enough so that the proofs could all be fully automated. We believe that the development of specialized 2-dimensional
rewriting techniques and tactics could greatly improve the formalization experience in \mathlib's framework for
bicategories. We refer the reader to our formalization for the computational details.

Theorem~\ref{exists-pseudo-bc-system} is surely well-known, but we could not find it precisely stated
in this way in the bicategorical literature (most entries tackle the problem of defining pseudofunctors out of the full bicategory of spans, but
not necessarily out of its pith). An \((\infty, 1)\)-categorical variant (stated in terms of cocartesian fibrations of \((\infty,1)\)-categories of spans) can be found as \cite[Th. 04AE]{kerodon}.

\section{Packaging symmetric monoidal categories as pseudofunctors}\label{section-packaging}
As the shape of the required data to produce a pseudofunctor out of
\(\mathrm{Pith}\left(\mathrm{Span}\left(\mathrm{Fin}\right)\right)\) is now clearly identified by Theorem~\ref{exists-pseudo-bc-system},
it remains to actually produce such data from the data of a symmetric monoidal category
\(\mathsf{C}\) to prove Theorem~\ref{main_th}.

\subsection{The Kleisli bicategory of symmetric lists}\label{subsection-kleisli-bicat}

Given a finite set \(J\) and a symmetric monoidal category \( \mathsf{C} \),
the product category \( \mathsf{C}^{J} \) that appears in Theorem~\ref{main_th} should be interpreted as the
category of symmetric monoidal functors from the free symmetric monoidal category on \( J \) to \( \mathsf{C} \), or
equivalently, according to Theorem~\ref{slist_fsmc}, as the category of symmetric monoidal functors
\( \mathrm{SList}(J) \to \mathsf{C} \).

Let \( J, K \) be finite sets and let \( g \) be a function \(J \to K \).
When inspecting the formula describing the functor \( g_! : \mathsf{C}^J \to \mathsf{C}^K \) that appears
in Theorem~\ref{main_th} and keeping in mind the interpretation of product categories as symmetric monoidal functors
from symmetric lists, it is reasonable to expect that the construction \( g \mapsto g_! \) and formulas about these
types of functors are instantiations in \( \mathsf{C} \) of formulas that hold universally in all symmetric monoidal categories,
i.e., formulas about symmetric lists.

Hence, to formalize the construction in Theorem~\ref{main_th}, it would be useful to be able to separate
things that happen at the universal level and to abstract the idea that the construction only instantiates
the universal formulas in \(\mathsf{C}\) as the last step.

For an ordinary algebraic theory, universal formulas are encoded by subcategories of free objects.
Equivalently, such formulas are encoded by the Kleisli categories for the corresponding monads:
the Kleisli category of a monad \(T\) on a 1-category \(\mathsf{C}\)
has by definition the same objects as \(\mathsf{C}\), and has morphisms from \(x\) to \(y\) defined as
morphisms \(x \to T(y) \) in \(\mathsf{C}\).
For instance, a universal formula for the theory
of commutative rings (where \(T := \mathbb{Z}[-]\)) is an equality of families of polynomials with integer coefficients.

Symmetric monoidal categories do not fit within this conceptual framework directly.
This is because, at heart, the structure of a symmetric
monoidal category is pseudoalgebraic: associativity and unitality only hold \emph{weakly}, i.e., up to specified natural isomorphisms rather
than up to equalities. It is much more natural to interpret a symmetric monoidal category as a \emph{pseudoalgebra}
for a \emph{pseudomonad}: these notions are the bicategorical analogues of the 1-categorical concepts of monads and algebras.
According to this interpretation, the appropriate object for encoding universal formulas would be a \emph{Kleisli bicategory}, a bicategorical analogue of the usual Kleisli category, for the pseudomonad corresponding to the theory of symmetric monoidal categories.

This interpretation turns out to be fruitful: the construction
\(I \mapsto \mathrm{SList}(I)\) with its canonical inclusion \(I \hookrightarrow \mathrm{SList}(I)\) and with the
universal property of symmetric lists that extends any function
\(J \to \mathrm{SList}(I)\) to a symmetric monoidal functor \(\mathrm{SList}(J) \to \mathrm{SList}(I)\), defines a
\emph{relative pseudomonad}~\cite[Def. 3.1]{Fiore2017} over the inclusion \(\mathrm{Set} \to \mathrm{Cat}\)\footnote{
  The fact that we have to use \emph{relative} pseudomonads here is somewhat artificial to our setup:
  the operation of sending a \emph{category} \(\mathsf{C}\) to the free symmetric monoidal category on \(C\) defines
  a \emph{bona fide} pseudomonad on \(\mathrm{Cat}\), but since we only formalized free symmetric monoidal categories
  on a set, and only established a link with symmetric lists in that setting, we are constrained to relative
  pseudomonads over the inclusion \(\mathrm{Set} \to \mathrm{Cat}\).}.
The Kleisli bicategory of a relative pseudomonad is also defined in~\cite{Fiore2017}. As we will
not need generalities on relative pseudomonads to state or formalize our results, we refer the reader to loc. cit. for more
details on the notion.

The interpretation of symmetric monoidal categories as pseudoalgebras motivates the following definition, which is the
specialization of~\cite[Th. 4.1]{Fiore2017} to our situation:
\begin{propdef}\label{kleisli-propdef}
  The \emph{Kleisli bicategory of symmetric lists} \( \mathcal{K}_{\mathrm{SList}} \) is the bicategory described by the
  following data:
  \begin{itemize}
    \item Objects are sets.
    \item Morphism categories \(I \rightsquigarrow J\) are product categories \(\left(\mathrm{SList}\left(J\right)\right)^I\).
    \item Identity morphisms \(I \rightsquigarrow I\) are given by the inclusions \(\iota_{I} : I \hookrightarrow
    \mathrm{SList}(I)\)
      that send an element \(i \in I\) to the singleton symmetric list.
    \item The composition of \(f : I \rightsquigarrow J \) with \(g : J \rightsquigarrow K\) is given by the composition
      \[\begin{tikzcd}[cramped]
          I & {\mathrm{SList}(J)} & {\mathrm{SList}(K)}
          \arrow["f", from=1-1, to=1-2]
          \arrow["{\Theta_g}", from=1-2, to=1-3]
      \end{tikzcd}\]
     where \( \Theta_g \) is the essentially unique extension of \( g \) to a symmetric monoidal functor.
    \item Given \(f, f' : I \rightsquigarrow J \), \(g, g' : J \rightsquigarrow K\), \(\varphi : f \to f' \) and
      \( \psi : g \to g' \), the horizontal composition of \(\varphi \) and \(\psi\) is the horizontal composition
      \( \Theta_{\psi} \cdot \varphi \)
      where \(\Theta_{\psi}\) is the unique monoidal natural transformation \(\Theta_g \to \Theta_{g'}\) induced by
      \(\psi\).
    \item Given \(f : I \rightsquigarrow J \), the right unitor \( f \circ \mathrm{Id}_I \iso f \) is the isomorphism
      \(\Theta_{f} \circ \iota_I \iso f \) that characterizes \(\Theta_f\) as the extension of \( f \) to symmetric lists.
    \item Given \(f : I \rightsquigarrow J \), the left unitor \( \mathrm{Id}_J \circ f \iso f \) is the isomorphism
      \(\Theta_{\iota_J} \circ f \iso f \) induced by the isomorphism
        \( \Theta_{\iota_J} \simeq \mathrm{Id}_{\mathrm{SList}(J)} \) that is deduced by the fact that both functors
        are monoidal and are isomorphic to the identity functor when restricted to singleton symmetric lists.
    \item Given \(f : I \rightsquigarrow J\), \(g : J \rightsquigarrow K\), \(h : K \rightsquigarrow L\), the associator
      is described by the diagram
      \[\begin{tikzcd}
              &&&&& {\mathrm{SList}(K)} \\
              I && {\mathrm{SList}(J)} \\
              &&&&& {\mathrm{SList}(L)}
              \arrow["{\Theta_h}", from=1-6, to=3-6]
              \arrow["f", from=2-1, to=2-3]
              \arrow["{\Theta_g}", from=2-3, to=1-6]
              \arrow[""{name=0, anchor=center, inner sep=0}, "{\Theta_{\Theta_h \circ g}}"', from=2-3, to=3-6]
              \arrow["\mu"', between={0.2}{1}, Rightarrow, from=0, to=1-6]
      \end{tikzcd}\]
      where \(\mu\) is the unique monoidal natural isomorphism whose component at a singleton symmetric list \([j]\)
      is the isomorphism
      \[\Theta_{\Theta_h \circ g}([j]) \iso \Theta_h(g(j)) \iso \Theta_h(\Theta_g([j])).\]
  \end{itemize}
\end{propdef}
All the proofs required to show that the data of Proposition~\ref{kleisli-propdef} indeed defines a bicategory structure
follow the exact same pattern:
the required equalities are equalities of families of morphisms of symmetric lists which are
components of monoidal natural transformations between monoidal functors out of some category of symmetric lists. It
thus suffices to show that these transformations are equal at components consisting of singleton symmetric lists. In this case,
the equalities reduce to equalities involving nothing but characterizing isomorphisms for \(\Theta\).
The fact that they suitably cancel each other in the diagrams is bookkeeping that Lean can discharge via the \lean{simp} tactic.

In ordinary 1-category theory, it is a classical fact that algebras for a monad \(T\) interpret as suitable
presheaves on the Kleisli category of the monad~\cite[Th. 1.2]{Day_1977}. This interpretation is obtained
by restricting the Yoneda embedding of the category of \(T\)-algebras along the canonical fully
faithful functor from the Kleisli category to the category of \(T\)-algebras.

Interpreting symmetric monoidal categories as pseudoalgebras for a (relative) pseudomonad, it is thus reasonable to
expect that symmetric monoidal categories interpret as \(\mathrm{Cat}\)-valued pseudofunctors out of the opposite of
the Kleisli bicategory of Proposition~\ref{kleisli-propdef}:
\begin{propdef}\label{exists-pseudo-kleisli}
  Let \( \mathsf{C} \) be a symmetric monoidal category. There is a pseudofunctor
  \[ \mathsf{C}_{\mathcal{K}} : \left(\mathcal{K}_{\mathrm{SList}}\right)^{\mathrm{op}} \to \mathrm{Cat} \]
  with the following properties:
  \begin{itemize}
    \item Given a set \(J\), \(\mathsf{C}_{\mathcal{K}}(J)\) is the category \( \mathsf{C}^J \).
    \item Given a morphism \(f : J \rightsquigarrow K \) in \( \mathcal{K}_{\mathrm{SList}} \), the functor
      \( \mathsf{C}_{\mathcal{K}}(f) \)
      is the composition
      \[\begin{tikzcd}
          {\mathsf{C}^K} & {\mathsf{C}^{\mathrm{SList}(K)}} & {\mathsf{C}^J}
          \arrow["{\Psi_{K, \mathsf{C}}}", from=1-1, to=1-2]
          \arrow["{f^*}", from=1-2, to=1-3]
      \end{tikzcd}\]
      where \(\Psi_{K, \mathsf{C}}\) is the functor that extends a function \( K \to \mathsf{C}\) into a symmetric monoidal functor
      \(\mathrm{SList}(K) \to \mathsf{C}\).
    \item Given morphisms \(f, f' : J \rightsquigarrow K \) in \( \mathcal{K}_{\mathrm{SList}} \) and a
      2-morphism \(\varphi : f \to f'\), the action of \(\mathsf{C}_{\mathcal{K}}\) on \( \varphi \) is
      described by the diagram
      \[\begin{tikzcd}
	{\mathsf{C}^K} & {\mathsf{C}^{\mathrm{SList}(K)}} && {\mathsf{C}^J}
	\arrow["{\Psi_{K, \mathsf{C}}}", from=1-1, to=1-2]
	\arrow[""{name=0, anchor=center, inner sep=0}, "{f^*}", curve={height=-18pt}, from=1-2, to=1-4]
	\arrow[""{name=1, anchor=center, inner sep=0}, "{(f')^*}"', curve={height=18pt}, from=1-2, to=1-4]
	\arrow["{\varphi^*}", between={0.2}{0.8}, Rightarrow, from=0, to=1]
      \end{tikzcd}.\]
    \item Given a set \(J\), the unitality isomorphism
        \[\mathsf{C}_{\mathcal{K}}(\mathrm{Id}_J) \iso \mathrm{Id}_{C^J}\]
      is such that its component at \(x : J \to \mathsf{C}\) is the isomorphism
      \[\iota_{J}^*\left(\Psi_{J, \mathsf{C}}(x)\right) \stackrel{\mathrm{def}}{=} \Theta_{x} \circ \iota_J \iso x \]
      that is part of the defining characterization of \(\Theta_{x}\) as the unique symmetric monoidal functor that
      extends \(x\) along \(\iota_J\).
    \item Given morphisms \(f : J \rightsquigarrow K \) and \(g : K \rightsquigarrow L\) in
      \(\mathcal{K}_{\mathrm{SList}}\),
      the composition isomorphism
        \[\mathsf{C}_{\mathcal{K}}(gf) \iso \mathsf{C}_{\mathcal{K}}(g) \circ \mathsf{C}_{\mathcal{K}}(f) \]
      is described by the diagram
      \[\begin{tikzcd}
              & {\mathsf{C}^{\mathrm{SList}(L)}} && {\mathsf{C}^{K}} \\
              {\mathsf{C}^L} &&& {\mathsf{C}^{\mathrm{SList}(K)}} & {\mathsf{C}^{J}} \\
              & {\mathsf{C}^{\mathrm{SList}(L)}}
              \arrow[""{name=0, anchor=center, inner sep=0}, "{g^*}", from=1-2, to=1-4]
              \arrow["{\Psi_{K, \mathsf{C}}}", from=1-4, to=2-4]
              \arrow["{\Psi_{L, \mathsf{C}}}", from=2-1, to=1-2]
              \arrow["{\Psi_{L, \mathsf{C}}}"', from=2-1, to=3-2]
              \arrow["{f^*}"', from=2-4, to=2-5]
              \arrow["{\Theta_{g}}"', from=3-2, to=2-4]
              \arrow["{\mu_{\mathsf{C}}}", between={0}{0.8}, Rightarrow, from=0, to=3-2]
      \end{tikzcd}\]
      where \(\mu_{\mathsf{C}}\) is a natural isomorphism such that the component at \(x : L \to \mathsf{C} \) is the
      unique monoidal isomorphism
      \[\Psi_{K, \mathsf{C}}\left(g^*\left(\Psi_{L, \mathsf{C}}(x)\right)\right) \stackrel{\mathrm{def}}{=}
      \Theta_{\Theta_{x} \circ g} \iso \Theta_{g} \circ \Theta_x \stackrel{\mathrm{def}}{=} \Theta_{g} \circ \Psi_{L,
    \mathsf{C}}(x) \]
      induced by the symmetric monoidal structures on both sides and by the family of isomorphisms
      \[
        \big(\Theta_{\Theta_{x} \circ g}\left([l]\right) \iso \Theta_{x}(g(l)) \iso
        \Theta_{x}(\Theta_g([l]))\big)_{l \in L}
      \]
      coming from the characterizing property of \(\Theta_{h}\) as the essentially unique monoidal functor out of
      symmetric lists extending a given function \(h\).
  \end{itemize}
    Furthermore, for every \(f : J \rightsquigarrow K\) in \(\mathcal{K}_{\mathrm{SList}}\), the functor
    \(\mathsf{C}_{\mathcal{K}}(f)\) carries a structure of a symmetric monoidal functor,
    induced by composition of symmetric monoidal structures on \(\Psi_{K, \mathsf{C}}\) and on \(f^*\).
    The action on 2-morphisms as well as the isomorphisms for composition and unitality are monoidal with respect
    to this structure.
\end{propdef}
\begin{remark}
  One feature of the pseudofunctor described in Proposition~\ref{exists-pseudo-kleisli} is that it
  relies exclusively on the universal property of symmetric lists as a free symmetric monoidal category. We could
  define a bicategory analogue to \(\mathcal{K}_{\mathrm{SList}}\) by replacing symmetric
  lists with any other model for free symmetric monoidal categories (e.g., the ``tautological'' one that we used
  in subsection~\ref{slist-vs-fsmc} as part of the proof that symmetric lists are free symmetric monoidal categories)
  and a similar construction would be possible.

  This is explained by the fact that the pseudofunctor described in Proposition~\ref{exists-pseudo-kleisli} is evidently the
  restriction of the bicategorical Yoneda embedding for the bicategory \(\mathrm{SymmMonCat}\) of symmetric monoidal
  categories, symmetric monoidal functors and monoidal natural transformations along a pseudofunctor
  \(\mathcal{K}_{\mathrm{SList}} \to \mathrm{SymmMonCat} \) that sends \(J\)
  to the category \(\mathrm{SList}(J)\). If one is willing to believe that \(\mathrm{SymmMonCat}\) is
  biequivalent to the bicategory of pseudoalgebras for the relative pseudomonad (in the sense of~\cite{arkor2025}) of
  symmetric lists, then the pseudofunctor in the above definition is exactly the restriction of
  the (bicategorical) Yoneda embedding along the canonical fully faithful functor from the Kleisli bicategory.
  While these generalities on pseudomonads may help in conceptually understanding
  the construction in Proposition~\ref{exists-pseudo-kleisli} and see how it is parallel to the 1-categorical phenomenon of~\cite[Th. 1.2]{Day_1977}, they are not strictly needed for our development, so we do not formalize them.
\end{remark}
The bulk of the \smcLinkRange{SymmMonCoherence/SList/Kleisli.lean}{317}{384}{formalization} of Proposition~\ref{exists-pseudo-kleisli} consists of constructing the \smcLinkRange{SymmMonCoherence/SList/Substitution.lean}{51}{102}{symmetric monoidal structure} on the monoidal extension functor
\(\Psi_{K, C} : \mathsf{C}^{K} \to \mathsf{C}^{\mathrm{SList}(K)}\), and proving that all the 2-morphisms that appear
(the action on 2-morphisms, as well as the unitality and composition isomorphisms) are monoidal with respect to
this structure. The monoidal structure on the extension functor is essentially provided by the fact that in a symmetric
monoidal category \(\mathsf{C}\), the tensor product bifunctor is itself symmetric monoidal.
Once this monoidal structure is obtained, one can construct all the 2-isomorphisms that appear in the proposition by applying
the universal property of symmetric lists as a free symmetric monoidal category.

The remaining proof obligations (compatibilities with associators and unitors) are all formalized in the same way as the ones
from Proposition~\ref{kleisli-propdef}: when unfolded and when working componentwise,
every formula that appears is an equality of components of some
\emph{monoidal} natural transformation between monoidal functors from a category of symmetric lists to \(\mathsf{C}\),
and once again, the proofs reduce to bookkeeping (and the characterizing properties of the associators and unitors of the Kleisli bicategory)
that automation tactics like \lean{simp} can discharge once we reduce to singleton symmetric lists.

As one would expect from its interpretation as a suitable Yoneda embedding, the encoding of symmetric monoidal
categories as pseudofunctors out of the opposite of the Kleisli bicategory is compatible with morphisms of symmetric
monoidal categories:
\begin{lemma}\label{unbias-lax}
Let \( \mathsf{C} \) and \( \mathsf{D} \) be symmetric monoidal categories and let \( F : \mathsf{C} \to \mathsf{D} \) be a
lax braided functor (i.e., a lax monoidal functor that respects the braiding).
Postcomposition with \( F \) extends to a lax natural transformation of pseudofunctors
\( \mathsf{C}_{\mathcal{K}} \to \mathsf{D}_{\mathcal{K}}\).
The naturality square for a morphism \( f : J \rightsquigarrow K \) is given by the diagram,
\[\begin{tikzcd}
	{\mathsf{C}^K} & {\mathsf{C}^{\mathrm{SList}(K)}} & {\mathsf{C}^J} \\
	{\mathsf{D}^K} & {\mathsf{D}^{\mathrm{SList}(K)}} & {\mathsf{D}^J}
	\arrow["{\Psi_{\mathsf{C}, K}}", from=1-1, to=1-2]
	\arrow["{F_*}"', from=1-1, to=2-1]
	\arrow["{f^*}", from=1-2, to=1-3]
	\arrow["{F_*}", from=1-2, to=2-2]
	\arrow[equals, from=1-3, to=2-2]
	\arrow["{F_*}", from=1-3, to=2-3]
	\arrow["\mu"', Rightarrow, from=2-1, to=1-2]
	\arrow["{\Psi_{\mathsf{D}, K}}"', from=2-1, to=2-2]
	\arrow["{f^*}"', from=2-2, to=2-3]
\end{tikzcd}\]
where the component at \(X : K \to \mathsf{C}\) of \( \mu \) is the unique
monoidal natural transformation \(\Theta_{F_*(X)} \to F_*(\Theta_X)\) such that
the component at the singleton symmetric list \([k]\) is the composite
\[\begin{tikzcd}
    {\Theta_{F_*(X)}([k])} & {F(X(k))} & {F(\Theta_{X}([k]))}
    \arrow["{\mathrm{def}}", from=1-1, to=1-2]
    \arrow["{F(\mathrm{def})}", from=1-2, to=1-3]
\end{tikzcd}\]

If \( F \) is (strong) monoidal, \(\mu\) is a natural isomorphism and thus the natural transformation
is a strong natural transformation of pseudofunctors.
\end{lemma}

We \smcLinkRange{SymmMonCoherence/SList/ToPseudofunctor/StrongTrans.lean}{86}{97}{implement} the above lemma in our
formalization, but unfortunately cannot go further than this and
cannot formalize yet that monoidal natural transformations of lax braided monoidal functors
induce modifications of the corresponding lax natural transformations:
at the time of writing, modifications in \mathlib are currently only defined between oplax natural
transformations or between strong natural transformations.

As the pseudofunctor \(\mathsf{C}_{\mathcal{K}}\) fully encodes the process of instantiating universal formulas of
symmetric lists in a given symmetric monoidal category \(\mathsf{C}\), we can now give a precise (and hence
formalizable) meaning to the idea that the formulas involved in Theorem~\ref{main_th} are all universal:
the pseudofunctor
  \[\mathsf{C}^{\otimes} : \mathrm{Pith}\left(\mathrm{Span}\left({\mathrm{Fin}}\right)\right) \to \mathrm{Cat}\]
we are trying to construct is actually the composition of \(\mathsf{C}_{\mathcal{K}}\) with a pseudofunctor
\[
  \Lambda : \mathrm{Pith}\left(\mathrm{Span}\left({\mathrm{Fin}}\right)\right) \to
    \mathcal{K}_{\mathrm{SList}}^{\mathrm{op}}.
\]
Constructing such a pseudofunctor \( \Lambda \) is hence our new main goal that we must achieve
in order to complete the proof of Theorem~\ref{main_th}.

\subsection{Duality for symmetric lists}~\label{subsubsection-slist-duality}
As an important step towards building a pseudofunctor \(\Lambda\) satisfying our requirements, we need to abstract
yet a few more intermediate constructions. As we can see from the informal description of \(\mathsf{C}^{\otimes}\) in
Theorem~\ref{main_th}, the covariant part of the pseudofunctor involves taking tensor products of families
indexed by the fibers of a function at a point.
It thus makes sense to abstract first the corresponding operation at the level of groupoids
of finite sets lying over a given set, before moving on to symmetric lists.
To this end, we abstract the following
\begin{propdef}\label{def-fibAt-Fin}
  Let \(J\) be a set and let \(j \in J\).
  The construction that sends a set \(X\) equipped with the function \(f : X \to J\) to \(f^{-1}\left(\{j\}\right)\)
  and sends a bijection of finite sets over \(J\)  to the restriction of the bijection to the fiber at \( j \) extends
  to a symmetric monoidal functor
  \(\mathrm{fib}_{j} : \mathrm{Core}\left(\mathrm{Fin}_{/J}\right) \to \mathrm{Core}(\mathrm{Fin}).\)
\end{propdef}
Taken in families, this operation realizes the classical correspondence between sets over a base and families of sets:
\begin{lemma}
  Let \(J\) be a finite set, the symmetric monoidal functor
  \[\mathrm{Fib}_{J} : \mathrm{Core}\left(\mathrm{Fin}_{/J}\right) \to \mathrm{Core}(\mathrm{Fin})^J \]
    induced by the family of symmetric monoidal functors \(\left(\mathrm{fib}_{j}\right)_{j \in J}\)
  extends to a symmetric monoidal equivalence of categories, where an inverse is given by the functor that
  sends a family of sets \((X_j)_{j \in J}\) to their disjoint union, equipped with the sum of the projections.
\end{lemma}

The \smcLinkRange{SymmMonCoherence/SList/Additive.lean}{751}{752}{formalization} of the last two constructions is rather
straightforward. In our implementation, we work with \lean{SList Unit} instead of
\(\mathrm{Core}(\mathrm{Fin})\), but this difference is purely cosmetic, as the coherence theorem shows they are
equivalent as symmetric monoidal categories, which is a fact we also include in our
\smcLinkRange{SymmMonCoherence/SList/EquivFintypeGrpd.lean}{524}{529}{formalization}.

Through the symmetric monoidal equivalence of categories \(\mathrm{I}_C\), the symmetric monoidal functor from last lemma unearths an
interesting construction about symmetric lists:
\begin{definition}
  Let \(J\) and \(K\) be finite sets. We define the \emph{duality functor}
  \[\mathcal{D}_{J, K} : \left(\mathrm{SList}(K)\right)^{J} \to \left(\mathrm{SList}(J)\right)^{K}\]
  as the following composition
  \begin{align*}
    \left(\mathrm{SList}(K)\right)^{J}
      &\iso \left(\mathrm{Core}\left(\mathrm{Fin}_{/K}\right)\right)^{J} &\\
      &\iso \left(\mathrm{Core}\left(\mathrm{Fin}\right)^K\right)^{J} &\\
      &\iso \left(\mathrm{Core}\left(\mathrm{Fin}\right)^J\right)^{K} &\textrm{by flipping variables}\\
      &\iso \left(\mathrm{Core}\left(\mathrm{Fin}_{/J}\right)\right)^{K} &\\
      &\iso \left(\mathrm{SList}(J)\right)^{K}. &\\
  \end{align*}
\end{definition}
This functor is a duality for symmetric lists: the composite equivalence
\[\mathrm{SList}(K)^J \iso \left(\mathrm{Core}\left(\mathrm{Fin}\right)^K\right)^{J}\] allows us to think of vectors
of \(K\)-valued symmetric lists as matrices of finite sets (at least heuristically), and the functor \(\mathcal{D}_{J, K}\) is essentially the
transposition operation in this interpretation. Note that as a composition of symmetric monoidal equivalences of
categories, \(\mathcal{D}_{J, K}\) is a symmetric monoidal equivalence of categories.

\subsection{Symmetric lists and multisets}\label{subsubsection-kleisi-multiset}

As we should make no assumptions on the ordering of elements of symmetric lists that go through the inverse of
\(\mathrm{I}_{C}\), a natural context to phrase some properties of \(\mathcal{D}_{J, K}\) is the language of multisets.

Recall that a \emph{multiset} of elements of \(K\) is a list up to permutation, i.e., a finite set with
multiplicities.
In \mathlib, given a type \lean{α}, the type \lean{Multiset α} is defined
as the quotient type of the type \lean{List α} by the relation \lean{Perm}, which is defined inductively in the core library of Lean and expresses that a list is a permutation of some other list.
The definition of \lean{Perm} makes it very apparent that symmetric lists are a categorification of multisets: more
precisely, the set of isomorphism classes of symmetric lists of elements of \(K\) is in bijection with multisets
of elements of \(K\), and so we may talk about the \emph{underlying multiset} of a symmetric list.
Multisets form a commutative
monoid where the addition is induced by concatenation of lists, and it is thus not hard to see that the multiset
underlying a tensor product of symmetric lists is the sum of the multisets underlying each factor.

We will be using set-like notations for multisets, keeping in mind that multiplicities are allowed
(so that \(\{i, i\}\) is a multiset distinct from \(\{i\}\)). Given a multiset \( M \) of elements of \(K\) and
\(k \in K\), we will denote by \(\mathfrak{c}(k, M)\) the multiplicity of \(k\) in \( M \). Given a symmetric list \(L\) in some set, we will denote by \(||L||\) its underlying multiset.

Continuing the analogy between 1-morphisms in \(\mathcal{K}_{\mathrm{SList}}\) and matrices of finite sets, given that isomorphism classes of
finite sets are described by natural numbers, we can expect that the underlying multisets of
components of 1-morphisms in \(\mathcal{K}_{\mathrm{SList}}\) behave like matrices of natural numbers,
and indeed, there is a matrix-like formula for the multiset underlying the composition:
\begin{lemma}\label{compute-comp-kleisli}
  Let \(J\), \(K\) and \(L\) be sets seen as objects of \( \mathcal{K}_{\mathrm{SList}} \), let
  \( f : J \rightsquigarrow K \), \( g : K \rightsquigarrow L \) be 1-morphisms in \(\mathcal{K}_{\mathrm{SList}}\)
  and let \(j \in J\), then if \(K\) is finite,
  \[ \lVert (g \circ f)(j) \rVert = \sum\limits_{k \in K}\mathfrak{c}\left(k, \lVert f(j)\rVert\right) \cdot
    \lVert g(k) \rVert. \]
\end{lemma}
\begin{proof}
  By inspecting the definition of composition in the Kleisli bicategory, one sees that it suffices to show that
  for every function \(f : K \to \mathrm{SList}(L)\) and for every symmetric list \(X \in \mathrm{SList}(K)\), the equality \[
    ||\Theta_f(X)|| = \sum\limits_{k \in K} \mathfrak{c}\left(k, ||X||\right) \cdot ||f(k)||
  \]
  holds. This formula turns tensor product of symmetric lists into additions of multisets, so it suffices to show it
  in the case where \(X\) is a singleton symmetric list \([k]\). In this case, the formula is tautological as
  \(\Theta_f([k])\) and \(f(k)\) are isomorphic by definition of \(\Theta\), and so have the same underlying multiset.
\end{proof}

Similarly, the multisets underlying the duality functor \(\mathcal{D}_{J, K}\) can be computed with matrix-like formulas,
cementing the interpretation of \(\mathcal{D}_{J, K}\) as a transposition operator.
\begin{lemma}\label{compute-duality}
  Let \(J\) and \(K\) be finite sets, and let \(\left(X_j\right)_{j \in J}\) be a family of symmetric lists of elements
  of \(K\). Let \(k \in K\), then
  \[ \lVert \mathcal{D}_{J, K}(X)(k) \rVert = \sum\limits_{j \in J} \mathfrak{c}(k, X(j)) \cdot \{j\}.\]
\end{lemma}
\begin{proof}
  Since two multisets are equal if and only if they have the same multiplicity at every element, the formula reduces to proving
  that
  \[\mathfrak{c}(j, \lVert\mathcal{D}_{J, K}(X)(k)\rVert) = \mathfrak{c}(k, \lVert X(j) \rVert)\] for every
  \(j \in J\) and every \(k \in K\).
  By unfolding the definition of \(\mathcal{D}_{J, K}\), it suffices to see that for every finite set \(S\) over \(J\),
  \(\mathfrak{c}(j, \mathrm{I}_{J}^{-1}(\mathrm{Fib}_J^{-1}(S)))\) is the cardinality of the preimage of \(j\) in \(S\),
  which we can see from the definitions.
\end{proof}

Because multisets are precisely isomorphism classes of symmetric lists, proving equalities of underlying multisets is a convenient computational
tool to assert the existence of \emph{some} isomorphism between symmetric lists.
We have seen already that there are situations where there might exist many such non-equal isomorphisms: the set of
automorphisms of the symmetric list that repeats the same element \(n\) times is the symmetric group \(\Sigma_n\),
as follows from Corollary~\ref{equiv_type_indx}. As we explain and formalize below, the presence of repeated elements in the
list is in fact the only possible reason for the existence of multiple isomorphisms between two symmetric lists.

\begin{definition}
  Let \(J\) be a set and let \(L\) be a symmetric list. We say that \(L\) is \emph{linear} if the underlying list
  of \(L\) has no duplicate.
\end{definition}
Equivalently, a symmetric list \(L\) is linear if all elements of its underlying multiset have multiplicity at most one.
Correspondingly, in our formalization, we introduce
\begin{leancode}
class Linear (L : SList C) where
  nodup : L.toList.Nodup
\end{leancode}
This notion recovers another classical form of the coherence theorem for symmetric monoidal categories:
\begin{lemma}\label{linear-iff}
  Let \(L_1\) and \(L_2\) be two symmetric lists with values in a set \(J\).
  \begin{enumerate}
    \item If there is a morphism \(L_1 \to L_2\), then \( L_2 \) is linear if and only if \( L_1\) is linear.
    \item If either \(L_1\) or \(L_2\) is linear, there is at most one morphism from \(L_1\) to \(L_2\).
  \end{enumerate}
\end{lemma}
\begin{proof}
  Since there is a morphism between two symmetric lists if and only if the underlying lists are permutations of each other, the
  first point follows from the fact that a list has no duplicate if and only if it is a permutation of a list that has
  no duplicates.

  For the second point, assuming that one of the lists is linear, by Corollary~\ref{equiv_type_indx}
  the set of morphisms from \( L_1 \) to \( L_2 \) is in bijection with bijections
  \[\phi : \{0, \ldots, \mathrm{length}(L_2) - 1\} \iso \{0, \ldots, \mathrm{length}(L_1) - 1\}\]
    such that \(L_2[\phi(i)] = L_1[i]\). Since the lists both have no duplicate (by the first point), if such a bijection \(\phi\) exists,
    the value of \(\phi(i)\) must be the unique index of \(L_1\) that has the value \(L_2[\phi(i)]\). Hence, there
    is at most one such bijection.
\end{proof}
In Lean, the conclusion of the second point of the above lemma is naturally expressed by the \lean{Subsingleton} class from
the core library of the proof assistant that expresses that a type has at most a single inhabitant.
Thus, the formalization of Lemma~\ref{linear-iff} takes the form of a registration of a
series of \lean{Subsingleton} instances:
\begin{leancode}
lemma linear_iff_of_iso {x y : SList C} (e : x ≅ y) : Linear x ↔ Linear y := by …
instance subsingleton_hom_of_linear_left (L L' : SList C) [h₀ : Linear L] :
    Subsingleton (L ⟶ L') where …
instance subsingleton_iso_of_linear_right (L L' : SList C) [Linear L'] :
    Subsingleton (L ≅ L') where …
instance subsingleton_hom_of_linear_right (L L' : SList C) [Linear L'] :
    Subsingleton (L ⟶ L') where …
instance subsingleton_iso_of_linear_left (L L' : SList C) [Linear L'] :
    Subsingleton (L' ≅ L) where …
\end{leancode}
Using this infrastructure lets us use the \lean{subsingleton} tactic from \mathlib in proofs,
which identifies the goal of a proof with
an equality between terms of a type with a \lean{Subsingleton} instance and closes said goal if Lean can infer this instance.

\subsection{Constructing \( \Lambda \)}

In subsection~\ref{subsubsection-slist-duality}, we constructed a duality for vectors of symmetric lists, which are the 1-morphisms in
\(\mathcal{K}_{\mathrm{SList}}\). Spans on a category \(C\) with pullbacks also have a canonical duality: the operation that sends a span
\(S : c \xleftarrow{f} \widehat{S} \xrightarrow{g} c'\) to its
transpose \({}^{t}S : c' \xleftarrow{g} \widehat{S} \xrightarrow{f} c\). Given \(f : c \to c'\),
this duality exchanges the spans \(f^*\) and \(f_!\). If one is willing to believe that the pseudofunctor
\(\Lambda: \mathrm{Pith}\left(\mathrm{Span}\left({\mathrm{Fin}}\right)\right) \to
  \mathcal{K}_{\mathrm{SList}}^{\mathrm{op}}\) we are trying to construct respects the dualities on its source and target, this leaves
only one possible definition:
\begin{thdef}\label{pbc-system-propdef}
  There is a unique \(\mathcal{K}_{\mathrm{SList}}^{\mathrm{op}}\)-valued Pith-Beck-Chevalley system on \(\mathrm{Fin}\) with the
  following properties:
  \begin{itemize}
    \item The pseudofunctor \(V : \mathrm{Fin}^{\mathrm{op}} \to \mathcal{K}_{\mathrm{SList}}^{\mathrm{op}} \) sends a set \(J\) to itself,
      and sends a function \(f : J \to K\) to the composition \[
        J \xrightarrow{f} K \xrightarrow{\iota_K} \mathrm{SList}(K).
      \]
    \item The pseudofunctor \(U : \mathrm{Fin} \to \mathcal{K}_{\mathrm{SList}}^{\mathrm{op}}\) sends a set \(J\) to itself,
      and sends \(f : J \to K\) to \(\mathcal{D}_{J,K}(V(f))\).
  \end{itemize}
\end{thdef}

The description in the theorem above might \emph{a priori} appear incomplete, as it does not spell out the
pseudofunctoriality data for \(V\) and \(U\) and
does not describe the required base change isomorphisms. In fact, as we will prove below, all of these isomorphisms are
uniquely determined by the values prescribed in Theorem~\ref{pbc-system-propdef}. We will use the notion of linear
symmetric lists introduced in~\ref{subsubsection-kleisi-multiset} and Lemma~\ref{linear-iff} to make this apparent.
\begin{proof}
Observe that, given a function \(f : J \to K\), the composition
\[ J \xrightarrow{f} K \xrightarrow{\iota_K} \mathrm{SList}(K) \]
is the function \(V(f)\) that sends \(j\) to the singleton symmetric list \([f(j)]\), which is linear.
This description provides a unique unitality isomorphism: the map \(\iota_J : j \mapsto [j]\) is the identity
morphism on \( J \) in  \(\mathcal{K}_{\mathrm{SList}}\), so that \(V(\mathrm{Id}_{J})\) is tautologically isomorphic to
the identity, and since the value at every \(j \in J\) is a linear symmetric list, this isomorphism must be unique by Lemma~\ref{linear-iff}.
Similarly, there is a unique composition isomorphism: we can see from the definition of
composition in \(\mathcal{K}_{\mathrm{SList}}\) that given \(f : J \to K\) and \( g : K \to L \),
the composition \(V(g) \circ V(f)\) is a function that sends \(j \in J\) to a symmetric list isomorphic to \([g(f(j))]\); uniqueness
is again ensured by the fact that the lists are linear.

Set \(U(f) := \mathcal{D}_{J,K}(V(f))\) and let \(k \in K\). Proposition~\ref{compute-duality}
lets us compute that
\[ \lVert U(f)(k) \rVert = \sum\limits_{j \in J} \mathfrak{c}(k, \lVert V(f)(j) \rVert) \cdot \{j\}, \]
and since by definition of \(V(f)\), \(\lVert V(f)(j)\rVert = \{f(j)\}\), we obtain that \(\mathfrak{c}(k, \lVert V(f)(j) \rVert)\)
is equal to \(1\) if \(k = f(j)\) and is zero otherwise. Hence, \begin{align}
\lVert U(f)(k) \rVert = f^{-1}\left(\{k\}\right) \label{U-eval}\end{align}
where the set \(f^{-1}\left(\{k\}\right)\) is seen as a multiset where all of its elements have multiplicity one. In particular,
\(U(f)(k)\) is linear as a symmetric list.
Similarly, given \(f : J \to K\) and \(g : K \to L\), we can use Lemma~\ref{compute-comp-kleisli} to compute that for
all \(l \in L\),
\begin{align*} \lVert (U(g) \circ U(f))(l) \rVert
    &= \sum\limits_{k \in K} \mathfrak{c}(k, \lVert U(g)(l) \rVert) \cdot \lVert U(f)(k) \rVert \\
    &= \sum\limits_{k \in K} \mathfrak{c}\left(k, g^{-1}\left(\{l\}\right)\right) \cdot f^{-1}\left(\{k\}\right)\\
    &= \sum\limits_{k \in g^{-1}\left(\{l\}\right)} f^{-1}\left(\{k\}\right)\\
    &= f^{-1}\left(g^{-1}(\{l\})\right).
\end{align*}
Hence, since both sides of the previous equality are pointwise linear symmetric lists with the same underlying multisets,
there is a unique isomorphism \[U(g \circ f) \iso U(f) \circ U(g)\] by Lemma~\ref{linear-iff}.
A similar argument shows that there is a unique isomorphism \(U(\mathrm{Id}_{J}) \iso \mathrm{Id}_{J}\).
Because of the linearity of all the symmetric lists involved, compatibility with the associators and unitors for \(U\) and \(V\) is automatically satisfied, as both sides of each required identity take place in sets with at most one element.

Defining the base change isomorphisms and proving that they respect vertical and horizontal composition make similar use
of linearity of symmetric lists: given a cartesian square
\[\begin{tikzcd}
  {I} & {J} \\
  {K} & {L}
  \arrow["{t}", from=1-1, to=1-2]
  \arrow["{l}"', from=1-1, to=2-1]
  \arrow["\lrcorner"{anchor=center, pos=0.125}, draw=none, from=1-1, to=2-2]
  \arrow["S"{anchor=center}, draw=none, from=1-1, to=2-2]
  \arrow["{r}", from=1-2, to=2-2]
  \arrow["{b}"', from=2-1, to=2-2]
\end{tikzcd}\]
of finite sets, we can compute using Lemma~\ref{compute-comp-kleisli}, Lemma~\ref{compute-duality} and equation~\eqref{U-eval} that on the one hand, for \(k \in K\),
\begin{align*}
  \lVert (U(r) \circ V(b))(k) \rVert &= \sum\limits_{l \in L} \mathfrak{c}(l, \{b(k)\}) \cdot r^{-1}(\{l\}) \\
                                     &= r^{-1}(\{b(k)\}) \\
\end{align*}
which is in particular linear. On the other hand
\begin{align*}
  \lVert (V(t) \circ U(l))(k) \rVert &= \sum\limits_{i \in I} \mathfrak{c}\left(i, l^{-1}\left(\{k\}\right)\right) \cdot
                                          \{t(i)\} \\
                                     &= \sum\limits_{i \in l^{-1}(\{k\})} \{t(i)\} \\
\end{align*}
and the multiplicity of \(j \in J\) in \(\lVert (V(t) \circ U(l))(k) \rVert\) is thus the cardinality of the subset of
elements \(i \in I\) such that \(l(i) = k\) and \(t(i) = j\). Since the square \(S\) is cartesian, this cardinality is at most
one and is exactly one when \(r(j) = b(k)\), so that the sum \(\sum\limits_{i \in l^{-1}(\{k\})} \{t(i)\}\) also describes
the set \(r^{-1}(\{b(k)\})\).

Hence, since both symmetric lists have equal underlying multisets, there is \textit{some} base change isomorphism
\(U(r) \circ V(b) \iso V(t) \circ U(l)\), and, as the symmetric lists involved
are linear, this base change isomorphism must in fact be unique by Lemma~\ref{linear-iff}.

Compatibility of this base change isomorphism with respect to vertical and horizontal pasting of pullback
squares is automatic, as well as its compatibility with vertical and horizontal unit squares.
Indeed all of these identities
can be expressed as a given base change isomorphism being equal to some other composition of morphisms. By the computations
above, the sources and targets of all base change isomorphisms are pointwise linear symmetric lists, hence by the second point
of Lemma~\ref{linear-iff}, there is a unique 2-morphism from the source to the target, so all of these identities must be true.
\end{proof}

The \smcLink{SymmMonCoherence/SList/ToPseudofunctor/Defs.lean}{87}{formalization} of this theorem follows the argument above.
Most of its content is carrying out the relevant multiset computations to prove that all the symmetric lists involved
implement the \lean{Linear} typeclass we introduced earlier. The compatibilities between 2-morphisms can then all be
discharged by using the already-mentioned \lean{subsingleton} tactic.

Proposition~\ref{pbc-system-propdef} finishes the proof of Theorem~\ref{main_th}, as the Pith-Beck-Chevalley system
defined in that proposition defines a pseudofunctor
\(\Lambda: \mathrm{Pith}\left(\mathrm{Span}\left({\mathrm{Fin}}\right)\right) \to
\mathcal{K}_{\mathrm{SList}}^{\mathrm{op}}\) using Theorem~\ref{exists-pseudo-bc-system}.
Finally, we can define the sought-after pseudofunctor unbiasing a symmetric monoidal category, and complete the
proof of Theorem~\ref{main_th}:
\begin{leancode}
abbrev Λ : Bicategory.Pith (Spans FintypeCat ⊤ ⊤) ⥤ᵖ Kleisliᵒᵖ := …
-- `Kleisli.pseudoOfSymmMonCat C` is the pseudofunctor called C_𝓚 in the text
universe v u in
def pseudoOfSymmMonCat
    (C : Type u) [Category.{v} C] [MonoidalCategory C] [SymmetricCategory C] :
    Bicategory.Pith (Spans FintypeCat.{0} ⊤ ⊤) ⥤ᵖ Cat.{v, u} :=
  Λ.comp (Kleisli.pseudoOfSymmMonCat C)
\end{leancode}
Note that in the listing above, to ensure good universe properties, we are using the lowest possible universe level for
the Kleisli bicategory and for the pith of the bicategory of spans of finite types. In Lean, there is a category of finite types \lean{FintypeCat.{u}}
for every universe \lean{u}, which corresponds to the universe in which the underlying types of the objects live in. Similarly,
there is in fact a Kleisli bicategory \lean{SList.Kleisli.{u}} of symmetric lists for every universe \lean{u}, and
the pseudofunctor \(\Lambda\) is in fact a family of pseudofunctors \lean{Λ.{u} : EffBurnsideFintype.{u} ⥤ᵖ Kleisli.{u}ᵒᵖ} for every universe \lean{u}. If
\lean{C : Type u} and \lean{I : Type u'} are types in different universes, the type of functions \lean{I → C} is only
guaranteed to be in \lean{Type (max u u')}. Hence, in general, the pseudofunctor \(\mathsf{C}_{\mathcal{K}}\) that
we defined in Proposition~\ref{exists-pseudo-kleisli} might introduce a universe bump if one performs its
construction using \lean{SList.Kleisli.{u}} for a monoidal category that is not \lean{u}-small. By picking \lean{u} to
be the lowest universe level \lean{0} for the category of finite types and for the Kleisli bicategory,
we ensure that no bump happens and that the original universe levels of the symmetric monoidal category on which
we perform the construction are preserved.
\section{Conclusion and future work}
To summarize, the pseudofunctor from Theorem~\ref{main_th} is obtained as a composition described by the diagram
\[\begin{tikzcd}
	{\mathrm{Pith}(\mathrm{Span}(\mathrm{Fin}))} && {\mathcal{K}^{\mathrm{op}}} && {\mathrm{Cat}}
	\arrow["\Lambda", from=1-1, to=1-3]
	\arrow["{\mathsf{C}_{\mathcal{K}}}", from=1-3, to=1-5]
\end{tikzcd}\]
where \(\mathcal{K}\) is the Kleisli bicategory for the relative pseudomonad that sends a set to the free symmetric
monoidal category on that set. The pseudofunctor \(\mathsf{C}_{\mathcal{K}}\) is the restriction to
free symmetric monoidal categories of the (bicategorical) Yoneda embedding for the bicategory of symmetric monoidal
categories, while the pseudofunctor \(\Lambda\) bundles the link between finite sets and free symmetric monoidal
categories. The version of the coherence theorem for symmetric monoidal categories that we implement allows one
to give first a better presentation of \(\mathcal{K}\) in terms of symmetric lists, and the complete description of
morphisms of symmetric lists given in Corollary~\ref{equiv_type_indx} gives the full link between bijections of
finite sets and symmetric lists.

As future work, we intend on formalizing some applications of the unbiasing that we constructed in this article.
We expect that our work will let us formalize tensor powers of objects in a
symmetric monoidal category. More precisely, we should obtain the assignment
\(x \mapsto x^{\otimes n}\) as a functor from \(\mathsf{C}\) to \(\Sigma_n\)-equivariant objects in \(\mathsf{C}\),
by first defining a functor \[\mathrm{B}{\Sigma_n} \to \mathrm{End}_{\mathrm{Pith}(\mathrm{Span}(\mathrm{Fin}))}(\{*\}, \{*\})\]
that sends the unique object to the set \(\{1,\ldots, n\}\) and then by letting \(\mathsf{C}^\otimes\) act on it to yield a
bifunctor \(\mathrm{B}{\Sigma_n} \to \mathsf{C}^\otimes(\{*\}) \to \mathsf{C}^\otimes(\{*\})\) on which we can flip
the order of the variables to produce the desired bifunctor.

Another possible corollary of our constructions is the ability to unbias commutative monoid objects internal to symmetric
monoidal categories: by interpreting a commutative monoid object \( x \) internal to a symmetric monoidal category
\( \mathsf{C}\) as a lax symmetric monoidal functor \( \{*\} \to \mathsf{C}\), our unbiasing of lax monoidal functors in
Proposition~\ref{unbias-lax} gives a lax natural transformation from (a pseudofunctor equivalent to)
the terminal pseudofunctor to \(\mathsf{C}_{\mathcal{K}}\).
When whiskerings of lax natural transformations by pseudofunctors are defined in \mathlib, this will give the
corresponding transformation of pseudofunctors out of \(\mathrm{Pith}(\mathrm{Span}(\mathrm{Fin}))\),
encoding operations of higher arities for the commutative monoid object \( x \).

Future work also includes formalizing the fact that the functor \(\Lambda\) that we construct in this work is
locally an equivalence of categories, providing extra evidence that
\(\mathrm{Pith}(\mathrm{Span}(\mathrm{Fin}))\) \emph{is} indeed the ``pseudoalgebraic
theory'' of symmetric monoidal ordinary categories, and fully formalizing the link between symmetric monoidal ordinary
categories and Cranch's point of view \cite{cranch2010}.
This can be seen from the fact that the function/fibration correspondence realizes an equivalence between the core of spans of
finite sets from \(I\) to \(J\) and
functions \(I \to \mathrm{Core}\left(\mathrm{Fin}_{/J}\right)\), the latter being equivalent to the category of 1-morphisms
\(I \rightsquigarrow J\) in \(\mathcal{K}\) via the coherence theorem. Hence, one needs to formalize a natural isomorphism
from the composition of the functor induced by \(\Lambda\) on 1-morphisms and that second equivalence to
the function/fibration correspondence. We expect no major roadblock for this part.

The local equivalence mentioned in the previous paragraph will bring us one step closer to what
we would consider the most definitive form of unbiasing: biequivalences of bicategories between the bicategory of pseudofunctors
(with lax, resp. oplax, resp. strong natural transformations as 1-morphisms)
\( \mathrm{Pith}(\mathrm{Span}(\mathrm{Fin})) \to \mathrm{Cat}\) that preserve products, and the bicategory of symmetric monoidal
categories with lax (resp. oplax, resp. strong) monoidal functors as 1-morphisms.
To achieve this, missing basic bicategorical constructions such as biequivalences, products in bicategories,
preservation of products by pseudofunctors, local characterization of biequivalences, etc., need to be formalized and added to \mathlib.

We expect to contribute the results from this work to \mathlib over time, as we fully believe that unbiased symmetric monoidal
categories are an essential step towards enabling more complex manipulations of objects in symmetric monoidal categories
(such as the aforementioned symmetric tensor powers in general monoidal categories) as well as a necessity to bridge ordinary
category theory with higher category theory once the latter develops in \mathlib.
Similarly, we expect to contribute the missing bicategorical infrastructure mentioned in the previous paragraph to \mathlib.

\printbibliography

\end{document}